\documentclass[11pt,english]{amsart}
\usepackage{amsmath}
\usepackage{amsfonts}
\usepackage{amssymb}
\usepackage{amsthm}
\usepackage{mathrsfs}
\usepackage{enumerate}
\usepackage[english]{babel}
\usepackage{xy,xypic}

\usepackage{xcolor}
\definecolor{rouge}{rgb}{0.85,0.1,.4}
\definecolor{bleu}{rgb}{0.1,0.2,0.9}
\definecolor{violet}{rgb}{0.7,0,0.8}
\usepackage[colorlinks=true,linkcolor=bleu,urlcolor=violet,citecolor=rouge]{hyperref}

\newtheorem{thm}{Theorem}[section]
\newtheorem{lemma}[thm]{Lemma}

\newtheorem{prop}[thm]{Proposition}
\newtheorem{cor}[thm]{Corollary}
\newtheorem{defi}[thm]{Definition}
\newtheorem{rem}[thm]{Remark}
\newtheorem{question}[thm]{Question}
\newtheorem{ex}[thm]{Example}

\newtheorem{theo_alpha}{Theorem}
\def\theo_alpha{\Alph{theo_alpha}}

\newtheorem{prop_alpha}{Proposition}
\def\theo_alpha{\Alph{theo_alpha}}

\def\theo_alpha{\Alph{theo_alpha}}

\newtheorem{quest_alpha}{Question}
\def\theo_alpha{\Alph{theo_alpha}}

\def\theo_alpha{\Alph{theo_alpha}}

\newtheorem{defi_alpha}{Definition}
\def\theo_alpha{\Alph{theo_alpha}}

\usepackage{fullpage}

\def\le{\leqslant}
\def\ge{\geqslant}

\def\J{\mathscr{J}}
\def\I{\mathcal{I}}
\def\bar#1{\overline{#1}}

\def\O{\mathcal O}      
    
\def\Z{\mathbb Z}

\def\N{\mathbb N}

\def\C{\mathbb C}

\def\P{\mathbb P}

\def\sl{\mathfrak{sl}}
\def\so{\mathfrak{so}}
\def\sp{\mathfrak{sp}}

\def\g{\mathfrak{g}}
\def\m{\mathfrak{m}}
\def\p{\mathfrak{p}}
\def\z{\mathfrak{z}}

\def\l{{\mathfrak{l}}}
\def\a{{\mathfrak{a}}}
\def\h{{\mathfrak{h}}}

\def\s{{\mathfrak{s}}}
\def\r{{\mathfrak{r}}}
\def\n{{\mathfrak{u}}}

\def\d{{\rm d}}

\def\Im{{\rm im}\,} 
\def\card{\mathrm{card}}
\def\reg{\mathrm{reg}}
\def\sing{\mathrm{sing}}
\def\O{\mathcal{O}}

\def\ad{{\rm ad}\,}

\def\bs{\boldsymbol}
\def\P{\mathscr{P}}
\def\RC{{\rm RC}}

\DeclareMathOperator{\Hom}{Hom}
\DeclareMathOperator{\Spec}{Spec}

\newcommand{\non}{\times}
\newcommand{\oui}{\surd}
\newcommand{\tvi}{\vrule depth7pt height13pt width0pt}
\newenvironment{Dynkin}{\setlength{\unitlength}{1.5pt}\begin{array}{l}}{\end{array}}
\newcommand{\Dbloc}[1]{\begin{picture}(20,20)#1\end{picture}}
\newcommand{\Dcirc}{\put(10,10){\circle{4}}}

\newcommand{\Deast}{\put(12,10){\line(1,0){8}}}
\newcommand{\Dwest}{\put(8,10){\line(-1,0){8}}}
\newcommand{\Dnorth}{\put(10,12){\line(0,1){8}}}
\newcommand{\Dsouth}{\put(10,8){\line(0,-1){8}}}

\newcommand{\Ddoubleeast}{\put(10,12){\line(1,0){10}}\put(10,8){\line(1,0){10}}}
\newcommand{\Ddoublewest}{\put(10,12){\line(-1,0){10}}\put(10,8){\line(-1,0){10}}}

\newcommand{\Dtext}[2]{\makebox(20,20)[#1]{\scriptsize $#2$}}

\newcommand{\Dskip}{\\ [-4.5pt]}
\newcommand{\Dspace}{\Dbloc{}}
\newcommand{\Dleftarrow}{\hskip-5pt{\makebox(20,20)[l]{\Large$<$}}\hskip-25pt}
\newcommand{\Drightarrow}{\hskip-5pt{\makebox(20,20)[l]{\Large$>$}}\hskip-25pt}

\begin{document}

\title{Jet schemes of the closure of nilpotent orbits}

\author[Anne Moreau]{Anne Moreau}
\address{Anne Moreau, Laboratoire de Math\'{e}matiques et Applications, 
T\'{e}l\'{e}port 2 - BP 30179, Boulevard Marie et Pierre Curie, 
86962 Futuroscope Chasseneuil Cedex, France}
\email{anne.moreau@math.univ-poitiers.fr}

\author[Rupert W.T. Yu]{Rupert W.T. Yu}
\address{Rupert W.T. Yu, Laboratoire de MathŽmatiques de Reims EA 4535,
U.F.R. Sciences Exactes et Naturelles,
Universit\'{e} de Reims Champagne Ardenne,
Moulin de la Housse - BP 1039,
51687 Reims cedex 2, France.} 
\email{rupert.yu@univ-reims.fr}


\date{\today}

\begin{abstract} 
We study in this paper the jet schemes of the closure of nilpotent orbits in a 
finite-dimensional complex reductive Lie algebra. 
For the nilpotent cone, which is the closure of the regular 
nilpotent orbit, all the jet schemes are irreducible. This was first 
observed by Eisenbud and Frenkel, and follows from a strong result 
of Musta\u{t}\c{a} (2001). Using induction and restriction of "little" nilpotent orbits 
in reductive Lie algebras, we show that for a large number of nilpotent orbits, 
the jet schemes of their closure are reducible. 
As a consequence, we obtain certain geometrical properties of these nilpotent orbit closures.  
\end{abstract}

\maketitle

\section{Introduction}
Throughout this paper, the ground field will be the field $\C$ of complex numbers.  
We shall work with the Zariski topology, and by {\em variety} we mean a reduced, irreducible 
and separated scheme of finite type over $\C$. 

For $X$ a scheme of finite type over $\C$ and $m\in\N$, we denote by $\J_m(X)$ 
the {\em $m$-th jet scheme of $X$}. It is a scheme of finite type over $\C$ whose
 $\C$-valued points are naturally in bijection with the $\C[t]/(t^{m+1})$-valued 
 points of $X$, cf.~e.g.,~\cite{Mu,EM,Is}.  We have $\J_0(X)\simeq X$ and 
 $\J_1(X)\simeq {\rm T}X$, where ${\rm T}X$ is the total tangent bundle of $X$; 
 see Section~\ref{S:jet} for more details about generalities on jet schemes. 
 From Nash \cite{Nas}, it is known that the geometry of the jet 
 schemes is deeply related to the singularities of $X$. As an illustration of that 
 phenomenon, we have the following result, first conjectured by Eisenbud and 
 Frenkel \cite[Introduction]{Mu}, which will be important for us. 
 
 \begin{theo_alpha}[{\cite[Thm.\,1]{Mu}}] \label{t:intro} 
 Let $X$ be an irreducible scheme of finite type over $\C$. If $X$ is locally a 
 complete intersection, then $\J_m(X)$ is irreducible for every $m\in\N$ if and only 
 if $X$ has rational singularities. 
\end{theo_alpha}  

According to Kolchin \cite{Kol}, in contrast to the above theorem, the arc space 
$\J_{\infty}(X)= \varprojlim\limits_{m} \J_m(X)$ of $X$ is always irreducible when $X$ is irreducible. 
In this paper, we shall be interested in the irreducibility of the jet schemes for 
the closure of nilpotent orbits in a complex reductive Lie algebra. 

Let $G$ be a complex connected reductive algebraic group, $\g$ its Lie 
algebra and $\mathcal{N}(\g)$ the nilpotent cone of $\g$. 
It is the subscheme of 
$\g$ associated to the augmentation 
ideal of $\C[\g]^G$. 
It is a finite union of nilpotent $G$-orbits, and there is a unique 
nilpotent orbit of $\g$, called the {\rm regular nilpotent orbit} 
and denoted by $\O_\reg$, such that 
$\mathcal{N}(\g)=\bar{\O_\reg}$. 

According to Kostant \cite{Ko}, 
the nilpotent cone is a complete intersection which 
is irreducible, reduced and normal. Furthermore, 
by \cite{He}, it has rational singularities.  
Hence by Theorem~\ref{t:intro}, 
the jet scheme $\J_m(\mathcal{N}(\g))$ is irreducible for every $m\ge 1$. In fact, by 
\cite[Prop.\,1.4 and 1.5]{Mu}, $\J_m(\mathcal{N}(\g))$ is also 
a complete intersection which is reduced for every $m\ge 1$.  

In \cite[Appendix]{Mu}, Eisenbud and Frenkel used these results  
to extend certain results of Kostant \cite{Ko} in the setting of jet schemes. 
In particular, they proved that $\C[\J_m(\g)]$ is free over the ring $\C[\J_m(\g)]^{\J_m(G)}$ 
of $\J_m(G)$-invariants of $\C[\J_m(\g)]$. 

Other nilpotent orbit closures do not share these geometrical properties in 
general. Indeed, according to a recent result of Namikawa \cite{Na}, 
for a nonzero and nonregular nilpotent orbit $\O$, $\bar{\O}$ is not a complete 
intersection. In addition, $\bar{\O}$ 
has not always rational singularities since it is not always normal, 
cf.~e.g.,~\cite{LeS,KP,Kr,Bro,So}. 

Thus, it is quite natural to ask the following question.  

\begin{quest_alpha} \label{q:intro} 
Let $\O$ be a 
nilpotent orbit of $\g$, and $m\in\N^*$. Is $\J_m(\bar{\O})$ irreducible?
\end{quest_alpha}

Answer Question \ref{q:intro} is the main purpose of this paper. 
For the zero orbit and the regular orbit, the answer is positive for every $m\in\N$. 
Outside these extreme cases, we will see that these jet schemes are rarely irreducible. 

\subsection*{Motivations}
Since $\bar{\O}$ is not 
a complete intersection for $\O$ nonzero and nonregular, 
Theorem \ref{t:intro} cannot be applied directly to 
answer Question~\ref{q:intro}. 
Very recently, Brion and Fu gave another proof of Namikawa's result, 
which is more uniform and slightly shorter, \cite{BF}.  
An interesting question, raised by Michel Brion to the first author, is 
whether jet schemes can be used to provide another proof of 
Namikawa's result. 

Let us explain how we can tackle this problem using jet schemes. 
Let $\O$ be a nilpotent orbit of $\g$. 
The singular locus of $\bar{\O}$
is exactly $\bar{\O} \setminus \O$. This follows from \cite[Lem.\,1.4]{Ka} and \cite{Pa}; 
see also \cite[Sec.\,2]{Hen} for a recent review. 
Moreover, we have 
$$
{\rm codim}_{\bar{\O}}(\bar{\O} \setminus \O)\ge 2.
$$  
For the nilpotent cone, we have precisely 
${\rm codim}_{\mathcal{N}(\g)}(\mathcal{N}(\g) 
\setminus \O_\reg)=2$, and the equality $\mathcal{N}(\g)_\reg=\O_\reg$ is a 
consequence of \cite[Thm.\,9]{Ko} (thus the notation $\O_\reg$ does not bear any confusion).  

So, if we assume that $\bar{\O}$ is 
a complete intersection, then $\bar{\O}$ is normal and so it has rational singularities 
by \cite{Hi} or \cite{Pa}. 
Hence, in that event, Musta\u{t}\c{a}'s Theorem implies that $\J_m(\bar{\O})$ is irreducible 
for every $m\ge 1$. So if we can show that  
$\J_m(\bar{\O})$ is reducible for some $m\ge 1$, then we would obtain a 
contradiction\footnote{There are other 
approaches to use jet schemes to show that $\bar{\O}$ is not a complete intersection; 
see Example \ref{ex:Nam}.}.  
The above was our original motivation to look into Question \ref{q:intro}.

It may happen that a variety $X$ is not a complete intersection, that $X$ 
has rational singularities and that nonetheless $\J_m(X)$ is irreducible for 
every $m\ge 1$. The cone over the Segre embedding 
$$\mathbb{P}^1 \times\mathbb{P}^{n-1} \hookrightarrow \mathbb{P}^{2n-1}, 
\qquad n \ge 2,$$ 
shows that this situation is possible, cf.~\cite[Ex.\,4.7]{Mu}. 
We do not know so far whether this situation may happen in the context 
of nilpotent orbit closures. 

More generally, following Nash's philosophy, it would be interesting to understand 
what kind of properties on the singularities of $\bar{\O}$ we can deduce from the study 
of $\J_m(\bar{\O})$, $m\ge 1$. 
The fact that $\bar{\O}$ is not a complete intersection (with $\O$ nonzero and 
nonregular) whenever $\J_m(\bar{\O})$ is reducible for some $m\geq 1$ is one 
illustration of such a phenomenon.

Nilpotent orbit closures also form an interesting family of varieties
providing examples and counter-examples in the context of jet schemes. 
For example, Examples \ref{ex2:2p} and \ref{ex:33} illustrate that the locally
complete intersection hypothesis cannot be removed from 
Lemma~\ref{l:irr},(3), and Theorem~\ref{t:Mus},(3).
Another example is that the normality is not conserved when we pass to jet schemes.
By Kostant, the nilpotent cone $\mathcal{N}(\g)$ is normal, and  
we show in Proposition~\ref{p:norm} that $\J_m(\mathcal{N}(\g))$, $m\ge 1$, is not normal  
for a simple Lie algebra $\g$. 

\subsection*{Main results} 
Let us describe the main techniques used to study Question \ref{q:intro}, and 
summarize the main results of the paper. 
To avoid technical details, we shall assume here that $\g$ is simple.


Let $X$ be an irreducible variety, and $m \in\N$. Then 
$\J_m(X)$ is irreducible if and only if 
$$
\pi_{X,m}^{-1}(X_{\sing}) \subset \bar{\pi_{X,m}^{-1}(X_{\reg})},
$$ 
where $\pi_{X,m} :\J_m(X) \to X$ is the 
canonical projection from $\J_m(X)$ onto $X$ (cf.~Section \ref{S:jet}), $X_\reg$ is the smooth part of $X$, and 
$X_{\sing}$ its complement 
(cf.~Lemma \ref{l:irr}). This is our starting point.  


For $\O$ a nilpotent orbit of $\g$, the singular locus of $\bar{\O}$ is $\bar{\O}\setminus 
\O$ (cf.~Section \ref{S:nil}). 
The above criterion leads us to the following two conditions  
which will be central in our paper (cf.~Definition \ref{d:RC}).  

\begin{defi_alpha} \label{d:intro}
Let $\O$ be a nilpotent orbit of $\g$. 
\begin{itemize} 
\item[1)] We say that $\O$ {\em verifies \RC$_1$} if $\pi_{\bar{\O},1}^{-1}(0)$ 
is not contained in the closure of $\pi_{\bar{\O},1}^{-1}(\O)$.   
\item[2)] Let $m\in\N^*$. 
We say that $\O$ {\em verifies \RC$_2(m)$} if for some nilpotent orbit $\O'$ 
contained in $\bar{\O}\setminus \O$, we have $\dim \pi_{\bar{\O},m}^{-1}(\O') 
\ge \dim \pi_{\bar{\O},m}^{-1}(\O)$. 
\end{itemize}
Here the letters \RC~stand for ``reducibility condition''.
\end{defi_alpha}

It follows readily (cf.~Lemma \ref{l:RC}) that 
if a nilpotent orbit $\O$ of $\g$ verifies \RC$_1$ (resp.\ \RC$_2(m)$ for some $m\in\N^*$), 
then $\J_1(\bar{\O})$ (resp.\ $\J_m(\bar{\O})$) is reducible.  

We have a characterization for the condition \RC$_1$ (cf.~Proposition \ref{p:eta}) 
which allows us for example to show that the nilpotent orbits 
of $\sl_{2p}(\C)$, with $p\ge 2$, associated with partitions of the form $(2^p)$ 
verify \RC$_1$ (cf.~Example~\ref{ex:2p}). Note that these orbits do not verify \RC$_2(1)$ 
(see again Example~\ref{ex:2p}). 

A nilpotent orbit $\O$ is called {\em little} if 
$0< 2\dim\O \le \dim\g $ (cf.~Definition \ref{d:litt}). 
For example, the minimal nilpotent orbit of $\g$ is little (cf.~Corollary \ref{c:min}), 
and the nilpotent orbits of $\sl_n(\C)$ associated with partitions of the form 
$(2^p,1^q)$, with $p,q\in\N^*$, are little (cf.~Example \ref{ex:litt}). There 
are many other examples (see Section \ref{S:litt}). 
Little nilpotent orbits verify 
both \RC$_1$ and \RC$_2(m)$ for every $m\in\N^*$ (cf.~Proposition \ref{p:litt}), 
and they turn out to be useful to study the  
reducibility of jet schemes of many other orbits via "restriction" or "induction" of orbits.

Firstly, by "restriction" to some Levi subalgebras of $\g$ (cf.~Proposition \ref{p:res}), 
we can obtain from nilpotent orbits $\O$ which verify $0< 2\dim\O < \dim\g$ 
examples of nilpotent orbits which verify \RC$_1$ (and that are not necessarily little); 
see Table \ref{tab:res}. More precisely, we have the following statement 
(cf.~Proposition~\ref{p:res}\,\footnote{Proposition \ref{p:res} is stated in a slightly more general 
context.}).

\begin{prop_alpha} \label{p2:intro}
Let $\l$ be a Levi subalgebra of $\g$ with a center of dimension one, and such that 
$\a:=[\l,\l]$ is simple. Denote by $A$ the connected subgroup 
of $G$ whose Lie algebra is $\a$. Let $e$ be a nilpotent element of $\a$ 
and suppose that the following conditions are satisfied: 
\begin{itemize}
\item[{\rm (i)}\;] $\a$ contains a regular semisimple element of $\g$, 
\item[{\rm (ii)}]  $2\dim G.e < \dim\g$. 
\end{itemize}
Then $A.e$ verifies \RC$_1$. 
\end{prop_alpha}


Secondly, by "induction", we can reach from nilpotent orbits of reductive Lie subalgebras of $\g$
many nilpotent orbits of $\g$. Here, we consider induction in the sense 
of Lusztig-Spaltenstein \cite{LS}. 
We refer the reader to Section \ref{S:ind} for the precise definition of a nilpotent 
orbit of $\g$ induced from another one in some proper Levi subalgebra $\l$ of $\g$. 
Our next statement says that condition \RC$_2(m)$, for $m\in\N^*$, passes through 
induction. 

\begin{theo_alpha} \label{t2:intro}
Let $\l$ be a Levi subalgebra of $\g$, $\O_\l$ a nilpotent orbit of $\l$ 
and $\O_\g$ the induced nilpotent orbit of $\g$ from $\O_\l$. If $\O_\l$ 
verifies \RC$_2(m)$ for some $m\in\N^*$, then $\O_\g$ also verifies \RC$_2(m)$. 
\end{theo_alpha}

From this result, we are able to deal with a large number of nilpotent orbits. 
First of all, any nilpotent orbit induced from a nilpotent orbit 
that has a little factor verifies \RC$_2(m)$ for every $m\in\N^*$ (cf.~Theorem~\ref{t2:ind}). 
In particular, if $\g$ is not of type $\bf{A_1}$, 
$\bf{B_2}=\bf{C_2}$ or $\bf{G_2}$, then the subregular nilpotent orbit $\O_{\rm subreg}$ of $\g$ 
verifies \RC$_2(m)$ for every $m\in\N^*$ (cf.~Corollary \ref{c:sub}),  
and so $\J_m(\bar{\O_{\rm subreg}})$ is reducible for every $m\in\N^*$.  


It turns out that many nilpotent orbits can be induced from a nilpotent orbit that has a little factor. 
This allows us to obtain the following result when $\mathfrak{g}$ is of type ${\bf A}$
(cf.~Theorem \ref{t:A}). 

\begin{theo_alpha} \label{t3:intro} 
Any nilpotent orbit of $\sl_n(\C)$ associated with a non rectangular 
partition of $n$ verifies \RC$_2(m)$ for every $m\in\N^*$.  
\end{theo_alpha}

For the other simple Lie algebras of classical types, we have the following (cf.~Theorem \ref{t:BCD}).

\begin{theo_alpha} \label{t4:intro}
Let $n\in \mathbb{N}^{*}$, $\bs{\lambda}=(\lambda_{1},\dots ,\lambda_{t})$ be a partition of $n$ and set $\lambda_{t+1}=0$.
Suppose that there exist $1\leqslant k < \ell \leqslant t$ such that $\lambda_{k} \ge \lambda_{k+1}+2$ and
$\lambda_{\ell} \ge \lambda_{\ell+1}+2$. 
\begin{itemize}
\item[1)] If $\O$ is a nilpotent orbit of $\so_{n}(\C )$ whose associated partition is $\bs{\lambda}$, then $\O$ 
verifies  \RC$_2(m)$ for every $m\in\N^*$. 
\item[2)] If $n$ is even and $\O$ is a nilpotent orbit of $\sp_{n}(\C )$ whose associated partition is $\bs{\lambda}$, then $\O$ 
verifies  \RC$_2(m)$ for every $m\in\N^*$. 
\end{itemize}
\end{theo_alpha}

While our result in the special linear case is exhaustive relative to induction, in the orthogonal and symplectic cases,
other nilpotent orbits can be obtained by induction from a little orbit (cf.~Theorem~\ref{t:BCD} and
Remark~\ref{rk:BCD}). For a simple Lie algebra of exceptional type, we have a list of nilpotent orbits which can be induced
from a little one (cf.~Appendix \ref{app:exc}).

\subsection*{Organization of the paper} 
In Section \ref{S:jet}, we state some basic properties on jet schemes with some proofs for the 
convenience of the reader. 

In Section \ref{S:nil}, we recall some standard properties of nilpotent orbit closures, and of their
jet schemes. We introduce here the two sufficient conditions \RC$_1$ and \RC$_2(m)$, $m\ge 1$,
to study the reducibility of these jet schemes, and 
we state some first properties of these conditions. 

Section \ref{S:litt} is devoted to little nilpotent orbits.
We show that little nilpotent orbits verify  
both \RC$_1$ and \RC$_2(m)$ for every $m\ge 1$, and we show how they can be used to prove 
condition \RC$_1$ via the "restriction" of orbits (cf. Proposition \ref{p:res}). 

In Section \ref{S:ind}, we study the induction of nilpotent orbits the sense of Luzstig-Spaltenstein, \cite{LS}.
The main result is that condition \RC$_2(m)$, for $m\ge 1$, 
passes through induction (cf.~Theorem \ref{t:ind}). 
We describe in Section \ref{S:csq} how to use Theorem \ref{t:ind} to obtain the reducibility of 
nilpotent orbit closures in simple Lie algebras according to their Dynkin type. 
The details of some of the conclusions 
are presented in Appendices \ref{app:stat} and \ref{app:exc}. 

We present in Section \ref{S:app} 
some applications of our results to geometrical properties of nilpotent orbit closures. 
We also discuss in this section some open problems.  

The standard notations relative to nilpotent orbits in classical simple Lie algebras 
are gathered together in Appendix~\ref{app:not}. 
Appendix~\ref{app:stat} contains some numerical data for 
classical simple Lie algebras, and Appendix~\ref{app:exc} summarizes our conclusions for 
simple Lie algebras of exceptional type. 

\subsection*{Acknowledgments} 
We thank Michel Brion for suggesting us this topic, and bringing to our attention Namikawa's 
result.  We are also indebted to {\em Mathematisches Forschungsinstitut Oberwolfach} (MFO) 
for its hospitality during our stay as "Research in pairs" in October, 2014.

We are grateful to Micha\"{e}l Bulois for pointing out an error in a previous version, 
and to Jean-Yves Charbonnel and Michel Brion for their comments.  

We are also thankful to the referee for careful reading and thoughtful suggestions.   

This research is supported by the ANR Project GERCHER Grant number ANR-2010-BLAN-110-02.

\tableofcontents

\section{Generalities on jet schemes} \label{S:jet}
In this section, we present some general facts on jet schemes. 
Our main references on the topic are \cite{Mu,EM,Is}, and  \cite[Chap.\,8]{DEM}. 

Let $X$ be a scheme of finite type over $\C$, and $m\in\N$. 

\begin{defi} \label{d:jet} 
An {\em $m$-jet of $X$} is a morphism
$$
\Spec \C[t]/(t^{m+1}) \longrightarrow X.
$$
The set of all $m$-jets of $X$ carries the structure of a scheme $\J_m(X)$, 
called the {\em $m$-th jet scheme of $X$}. 
It is a scheme of finite type over $\C$ characterized by the following functorial property: 
for every scheme $Z$ over $\C$, we have  
$$
\Hom(Z,\J_m(X)) = \Hom(Z\times_{\Spec\C} \Spec \C[t]/(t^{m+1}),X).
$$
\end{defi}
The $\C$-points of $\J_m(X)$ are thus the $\C[t]/(t^{m+1})$-points of $X$. 
From Definition \ref{d:jet}, we have for example that $\J_0(X) \simeq X$ and that 
$\J_1(X)\simeq {\rm T}X$ where ${\rm T}X$ denotes the total tangent bundle of $X$.

For $p\in\{0,\ldots,m\}$, the canonical projection $\C[t]/(t^{m+1})\to \C[t]/(t^{p+1})$ 
induces a {\em truncation morphism},
$$
\pi_{X,m,p} : \J_m(X) \rightarrow \J_p(X).
$$ 
We shall simply denote by $\pi_{X,m}$ 
the morphism $\pi_{X,m,0}$, 
$$
\pi_{X,m} : \J_m(X) \rightarrow \J_0(X)\simeq X.
$$ 
Also, the canonical injection $\C \hookrightarrow \C[t]/(t^{m+1})$ 
induces a morphism $\iota_{X,m} : X \to \J_m(X)$, and we have 
$ \pi_{X,m} \circ  \iota_{X,m}={\rm Id}_X$. Hence $ \iota_{X,m}$ 
is injective and  $\pi_{X,m}$ is surjective. We shall always view $X$ as 
a subscheme of $\J_m(X)$. 

If $f :X \to Y$ is a morphism of schemes, then we naturally obtain a morphism 
$f_m :\J_m(X) \to \J_m(Y)$ making the following diagram commutative,

\begin{center}
~\xymatrix{\J_m(X) \ar[r]^{f_m}\ar[d]_{\pi_{X,m}} & \J_m(Y) \ar[d]^{\pi_{Y,m}}\\
X \ar[r]_{f} & Y }
\end{center}

\begin{rem} \label{rk:aff} 
In the case where $X$ is affine, we have the following 
explicit description of $\J_m(X)$. 

Let $n\in \N^{*}$ and $X \subset \C^n$ be the affine subscheme defined by an ideal $I =(f_1,\ldots,f_r)$
of $\C[x_1,\ldots,x_n]$. Thus
$$
X =\Spec\,\C[x_1,\ldots,x_n]/I.
$$
For $k\in\{1,\ldots,r\}$, we extend $f_k$ as a map from $(\C[t]/(t^{m+1}))^n$ to $\C[t]/(t^{m+1})$ 
via base extension. Then giving a morphism
$\gamma: \Spec\C[t]/(t^{m+1}) \to X$ is equivalent to giving a morphism 
$\gamma^*: \C[x_1,\ldots,x_n]/I \to \C[t]/(t^{m+1})$, 
or  to giving 
$$
\gamma^*(x_i) =\sum\limits_{j=0}^m \gamma_{i}^{(j)} t^j \qquad (1\le i\le n)
$$
such that for any $k\in\{1,\ldots,r\}$, 
$$
f_k(\gamma^*(x_1),\ldots,\gamma^*(x_n)) =0 \; \text{ in }\; \C[t]/(t^{m+1}).
$$
For $k\in\{1,\ldots,r\}$, 
there exist functions $f_k^{(0)},\ldots,f_k^{(m)}$, which depend only on $f$, 
in the variables $\boldsymbol{\gamma} = (\gamma_{i}^{(j)})_{\genfrac{}{}{0pt}{}{1 \le i \le n,}{0 \le j \le m}}$
such that 
\begin{eqnarray} \label{eq;aff}  
f_k\bigl(\gamma^*(x_1),\ldots,\gamma^*(x_n)\bigr) = \sum_{j=0}^{m} 
f_k^{(j)} ( \boldsymbol{\gamma} ) \, t^j.
\end{eqnarray}  
The jet scheme $\J_m(X)$ is then the closed subscheme in $\C^{(m+1)n}$ defined by the ideal generated by the
polynomials $f_k^{(j)}$, where $k\in\{1,\ldots,r\}$ and $j\in\{0,\ldots,m\}$. More precisely,
$$
\J_m(X) \simeq \Spec \C [ x_1^{(j)},\ldots,x_n^{(j)} ; j=0,\ldots, m ]/(f_k^{(j)} \, ;\, k=1,\ldots,  r, \, j=0,\ldots,m).
 $$ 

In particular, if $X$ is an $n$-dimensional vector space, then 
$\J_m(X) \simeq \C^{(m+1)n}$ and for $p\in \{0,\ldots,m\}$, 
the projection $\J_m(X) \to  \J_p(X)$ corresponds to the projection 
onto the first $(p+1)n$ coordinates. 
\end{rem}

\begin{ex} \label{ex:aff} 
Let us consider a concrete example. 
Let $X= \Spec\C[x,y,z]/(x^2+yz) \subset \C^{3}$ and let us compute 
$\J_1(X)$ and $\J_2(X)$. We have 
$$
(x_0+x_1 t+x_2 t^2)^2+(y_0+y_1t+y_2 t^2)(z_0+z_1 t+z_2 t^2)
$$
$$
\hspace{1cm}= x_0^2 + y_0z_0 + (2 x_0 x_1 +y_0 z_1+ y_1 z_0 ) t+ (2 x_0 x_2 + x_1^2 
+y_0 z_2 + y_2 z_0 + y_1 z_1 )t^2 \mod t^3. 
$$
Hence $\J_1(X)$ is the subscheme of 
$$
 \J_1(\C^3) \simeq  \C[x_0,y_0,z_0,x_1,y_1,z_1] 
 $$
 defined by the ideal 
$$
(x_0^2 + y_0z_0, 2 x_0 x_1 +y_0 z_1+ y_1 z_0), 
$$ 
and $\J_2(X)$ is the subscheme of 
$$
\J_2(\C^3) \simeq \C[x_0,y_0,z_0,x_1,y_1,z_1,x_2,y_2,z_3] 
$$
defined by the ideal 
$$
(x_0^2 + y_0z_0, 2 x_0 x_1 +y_0 z_1+ y_1 z_0, 
2 x_0 x_2 + x_1^2 +y_0 z_2+y_1 z_1 + y_2 z_0).
$$ 
\end{ex}

We now list some basic properties that we need in the sequel. Their proofs can found in 
\cite[Lem.\,2.3, Rem.\,2.8, Rem.\,2.10]{EM}.

\begin{lemma}  \label{l:jet}
\begin{itemize}
\item[1)]  For every open subset $U$ of $X$, we have $\J_m(U) = \pi_{X,m}^{-1}(U)$.  
\item[2)] For every scheme $Y$, we have a canonical 
isomorphism $\J_m(X\times Y)\simeq \J_m(X)\times \J_m(Y).$ 
\item[3)] If $G$ is a group scheme over $\C$, then $\J_m(G)$ is also a group 
scheme over $\C$. Moreover, if $G$ acts on $X$, then $\J_m(G)$ acts on 
$\J_m(X)$. 
\item[4)] If $f : X\to Y$ is a smooth surjective morphism between schemes, then 
$f_m$ is also smooth and surjective for every $m\in \N^{*}$. 
\end{itemize}
\end{lemma}

\subsection*{Geometrical properties} 
It is known that the geometry of the jet schemes $\J_m(X)$, $m\ge 1$,  
is closely linked to that of $X$. More precisely, we can transport 
some geometrical properties from $\J_m(X)$ to $X$. 

The following proposition gives examples of such phenomena
(\cite{MFK} and \cite[Thm.\,3.5]{Is}). 

\begin{prop} \label{p:geo} Let $m\in \N^*$. 
If $\J_m(X)$ is smooth (resp., irreducible, 
reduced, normal, locally a complete intersection) for some $m$, then so is $X$. 
\end{prop}
For smoothness, the converse is true, even with "every $m$" instead 
of "for some $m$". In fact, for smooth varieties, we have the following 
more precise statement, \cite[Cor.\,2.12]{EM}. 

\begin{prop} \label{p:smo}
If $X$ is a smooth variety of dimension $n$, then the truncation morphism  
$\pi_{m,p}$, for $p\in\{0,\ldots,m\}$, is a locally trivial projection with fiber 
isomorphic to $\C^{(m-p)n}$. In particular, $\J_m(X)$ is a smooth variety of dimension 
$(m+1)n$.  
\end{prop}

For the other properties stated in Proposition \ref{p:geo}, the converse is not true in general. 
We refer to \cite[\S3]{Is} for counter-examples. 
We shall encounter other counter-examples in this 
paper in the setting of nilpotent orbit closures. 
In this setting, our main purpose is to study the irreducibility of jet schemes. 
The following lemma gives a necessary and sufficient condition for 
the converse of Proposition \ref{p:geo} to hold for irreducibility. 

We denote by $X_\reg$ the smooth part of $X$, and by $X_\sing$ its complement.

\begin{lemma} \label{l:irr} 
Assume that $X$ is an irreducible reduced scheme of finite type over $\C$, and let $m\in\N^*$.  
\begin{itemize}
\item[1)] $\bar{\pi_{X,m}^{-1}(X_\reg)}$ is an irreducible 
component of $\J_m(X)$. 
\item[2)] $\J_m(X)$ is irreducible if and only if $\pi_{X,m}^{-1}(X_\sing)$ 
is contained in $\bar{\pi_{X,m}^{-1}(X_\reg)}$. 
\item[3)] If $X$ is a complete intersection, then  
$\J_m(X)$ is irreducible if and only if $\dim \pi_{X,m}^{-1}(X_\sing) < \dim 
\pi_{X,m}^{-1}(X_\reg)$. 
\end{itemize}
\end{lemma}

In particular, if $\dim \pi_{X,m}^{-1}(X_\sing) \ge \dim \bar{\pi_{X,m}^{-1}(X_\reg)}$, 
then $\J_m(X)$ is reducible.

\begin{proof} 
Part (3) is proved in \cite[Prop\,.1.4]{Mu}, and its proof implies parts (1) and (2). 
More precisely, since $X_\reg$ is smooth and irreducible, $\bar{\pi_{X,m}^{-1}(X_\reg)}$ is an 
irreducible closed subset of $\J_m(X)$ of dimension $(m+1)\dim X$, cf.~Proposition \ref{p:smo}.
Then parts (1) and (2) follow easily from the fact that we have the decomposition
$$
\J_m(X) = \pi_{X,m}^{-1}(X_\sing) \cup \bar{\pi_{X,m}^{-1}(X_\reg)}
$$
of closed subsets, and that $\pi_{X,m}^{-1}(X_\sing)  \not \supset \bar{\pi_{X,m}^{-1}(X_\reg)}$.
\end{proof}


There are also subtle connections between the geometry of $\J_m(X)$, $m\ge 1$,  
and the singularities of $X$ which are important for us. In particular, 
according to \cite[Thm.\,0.1, Prop.\,1.5 and 4.12]{Mu}, 
we have:

\begin{thm}[{Musta\c{t}\u{a}}] \label{t:Mus} 
Let $X$ be an irreducible variety over $\C$. 
\begin{itemize} 
\item[1)] If $X$ is locally a complete intersection, then $\J_m(X)$ 
is irreducible for every $m\ge 1$ if and only if $X$ has rational singularities. 
\item[2)] 
If $X$ is locally a complete intersection and if $\J_m(X)$ is irreducible for 
some $m\ge 1$, then $\J_m(X)$ is also reduced. 
\item[3)]  If $X$ is locally a complete intersection, 
then $(\J_1(X))_\reg=\pi_{X,1}^{-1}(X_\reg)$. 
\end{itemize}
\end{thm}

Let us give an easy counter-example to the converse implication of Proposition \ref{p:geo} for normality. 
This example turns out to be a particular case of a more general 
situation that will be studied in Proposition \ref{p:norm}.

\begin{ex} \label{ex:sl2} 
Let $X$ be as in Example \ref{ex:aff}. Then $X$ is a complete intersection and 
it is normal since the singular locus is reduced to $\{0\}$ which has codimension 
2 in $X$. Next, it is not difficult to verify that $\J_1(X)$ is irreducible, reduced and 
that it is a complete intersection. But $\J_1(X)$ is not 
normal. Indeed, by Theorem \ref{t:Mus},(3), 
$$(\J_1(X))_\sing=\pi_{X,1}^{-1}(\{0\}) \simeq \{0\}\times \C^3.$$ 
Hence, the singular locus of $\J_1(X)$ has codimension 1 in $\J_1(X)$ 
since $\dim \J_1(X)=2\dim X = 4$.    
\end{ex}

\subsection*{Group actions} 
Let $G$ be a connected algebraic group, acting on a variety $X$, and $m\in\N$. 
Denote by 
$$
\rho\colon G\times X\to X, \quad (g,x)\mapsto g.x
$$ 
the corresponding action. As stated in Lemma \ref{l:jet}, the morphism 
$$
\rho_m : \J_m(G\times X)\simeq \J_m(G) \times \J_m(X) \to 
\J_m(X)$$ 
defines an action of $\J_m(G)$ on $\J_m(X)$.

Recall that we embed $X$ into $\J_m(X)$ through $\iota_{X,m}$. 
For $x \in X$, let us denote by $G^x$ the stabilizer of $x$ in $G$, and for $m\in \N$, 
we denote by $\J_m(G)^x$ its stabilizer in $\J_m(G)$. The following results are 
probably standard. Since we have not found any reference, we shall include their proofs. 

\begin{lemma}  \label{l:gpe}
Let $x \in X$. Then, 
$$\J_m(G).x= \J_m(G.x), \quad 
\J_m(G^x) = \J_m(G)^x \quad  \text{ and }\quad 
\pi_{\bar{G.x},m}^{-1}(G.x) =\J_m(G.x).$$ 
\end{lemma}

\begin{proof} 
The morphism $G \times \{x\} \to G.x, \, (g,x) \mapsto g.x$ 
is a submersion at all points of $G \times \{x\}$. 
Hence, according to \cite[Ch.\,III, Prop.\,10.4]{Ha}, it is a smooth morphism 
onto $G.x$. So, by Lemma~\ref{l:jet},(4), the induced morphism $\J_m(G) 
\times \{x\} \to \J_m(G.x)$ is also smooth and surjective. 
Consequently, we have the first equality $\J_m(G).x=\J_m(G.x)$. 

By applying the first equality to the algebraic group $G^x$, we get 
$\J_m(G^x).x=\J_m(G^x.x)$, and whence the inclusion $\J_m(G^x)\subset \J_m(G)^x$. 

Conversely, let $\gamma :\Spec\,\C[t]/(t^{m+1}) \to G$ be an element of $\J_m(G)^x$. 
Then $\rho_m(\gamma,x)=x$, and hence viewing $x$ as a morphism $x : \Spec\,\C[t]/(t^{m+1}) \to X$,
we have
$$\rho(\gamma(\tau),x(\tau))=x(\tau)$$ 
where $\tau$ is the unique element of $\Spec\,\C[t]/(t^{m+1})$. 
Thus $\gamma(\tau)\in G^x$ and $x(\tau)=x$. So we have $\gamma \in \J_m(G^x)$,
and the second equality follows.

The third equality is a direct consequence of Lemma \ref{l:jet},(1) since $G.x$ 
is open in its closure. 
\end{proof}

Let $\g$ be the Lie algebra of $G$. We
consider now the adjoint action of $G$ on $\g$. 
For the results we present here, we refer the reader to \cite[Appendix]{Mu}. 
Denote by 
$$\g_m := \g\otimes_\C \C[t]/(t^{m+1})$$ 
the {\em generalized Takiff Lie algebra} whose Lie bracket 
is given by 
$$[u \otimes x(t), v \otimes y(t)] =[u,v] \otimes x(t)y(t) \qquad (u,v \in \g,\; 
x(t), y(t) \in\C[t]/(t^{m+1}) ).$$
As Lie algebras, we have 
$$\J_m(\g) \simeq\g_m \simeq {\rm Lie}(\J_m(G)).$$ 
In the sequel, when there is no confusion, we shall use the notations 
$\g_m$ and $G_m$ for $\J_m(\g)$ and $\J_m(G)$ respectively. If $\a$ is a Lie subalgebra of 
$\g$, then $\J_m(\a)\simeq \a_m$ is a Lie subalgebra of $\g_m$. In particular, 
for $x\in\g$, we have $(\g_m)^x=(\g^x)_m$, where for any subalgebra $\m$ 
of $\g_k$, with $k\ge 0$, $\m^x$ stands for the centralizer of $x$ in $\m$. 

We can identify $\g_m$ with $\g^{m+1} \simeq \J_m(\g)$ as 
a variety through the map
$$\g^{m+1} \to \g_m, \quad (x_0,x_1,\ldots,x_m)\mapsto x_0 + x_1\otimes t +\cdots 
+ x_m \otimes t^m.$$
Let $G_m$ be a connected algebraic group whose Lie algebra is $\g_m$. 
Let $\C[\g_m]$ be the coordinate ring of $\g_m$, and let $\C[\g_m]^{G_m} $ 
be the subring of $G_m$-invariants. We conclude in this section with  
the following result. 

\begin{lemma} \label{l:inv} 
For $f\in \C[\g]^G$, 
the polynomials $f^{(0)},\ldots,f^{(m)}$, as defined in Remark \ref{rk:aff}, are 
elements of $\C[\g_m]^{G_m}$.  
\end{lemma}
\begin{proof}
This is straightforward from the explicit description of the polynomials 
$f^{(0)},\ldots ,f^{(m)}$ given in Remark \ref{rk:aff}.
\end{proof}


\section{Nilpotent orbit closures} \label{S:nil}
From now on, we let $G$ to be a connected reductive algebraic group 
over $\C$, $\g$ its Lie algebra and $\mathcal{N}(\g)$ the nilpotent 
cone of $\g$. Recall that $\mathcal{N}(\g)$ is the 
subscheme of $\g$ defined by the augmentation ideal of $\C[\g]^G$, and that  
$\mathcal{N}(\g)=\bar{\O_\reg}$ where $\O_\reg$ is the regular nilpotent 
orbit of $\g$ (cf.~Introduction).  
As mentioned in the Introduction, we are interested in this paper in 
the irreducibility of jet schemes of the closure of nilpotent orbits.

Recall that for an arbitrary nilpotent orbit $\O$ of $\g$, the singular locus of $\bar{\O}$
is $\bar{\O} \setminus \O$ and that 
${\rm codim}_{\bar{\O}}(\bar{\O} \setminus \O)\ge 2$ (cf.~Introduction). 
%
%

\begin{defi}  \label{d:gO}
Let $\O$ be a nonzero nilpotent orbit of $\g$. 
Define $\g_\O$ to be the smallest semisimple ideal of $\g$ 
containing $\O$. 
\end{defi}

More precisely, if $\g\simeq \z(\g) \times \s_1\times\cdots\times\s_m$, 
with $\z(\g)$ the center of $\g$ and $\s_1,\ldots,\s_m$ the simple 
factors of $\g$, then $\O=\O_1\times \cdots\times\O_m$, with $\O_i$ a nilpotent 
orbit of $\s_i$ for $i=1,\ldots,m$,  
and 
$$
\g_\O= \s_{i_1}\times\cdots\times \s_{i_k}
$$ 
where $\{i_1,\ldots,i_k\}$ is the set of integers $j\in \{1,\ldots,m\}$ 
such that $\O_j$ is nonzero. 
In particular, if $\O$ is zero, then $\g_\O=0$, and if $\O$ is nonzero 
and $\g$ is simple, then $\g_\O=\g$. 
 
For $\O$ a nilpotent orbit of $\g$, we denote by  
$\I_{\bar{O}}$ the defining ideal of $\bar{\O}$ in $\g_{\O}$. 
Thus,
$$\bar{\O}={\rm Spec}\,\C[\g_\O]/\I_{\bar{\O}}.$$
Recall that $\bar{\O}$ is conical, so $\I_{\bar{\O}}$ is a homogeneous ideal.

\begin{lemma} \label{l:deg2}
Let $\O$ be a nonzero nilpotent orbit of $\g$.  
If $f_1,\ldots,f_s$ are homogeneous generators of $\I_{\bar{\O}}$,  
then the minimum degree of the $f_i$'s is exactly 2. 
\end{lemma}

\begin{proof}
By the above discussion, $\O$ is a product of nilpotent orbits. We may therefore assume that 
$\g=\g_\O$ is simple. 

Assume that for some $i\in\{1,\ldots,s\}$, $\deg f_i=1$. 
A contradiction is expected. 
Let $\mathcal{V}$ be the intersection of all the hyperplanes $\mathcal{H}_{g}$, 
$g\in G$, defined by the linear form 
$$
g.f_i \colon \g \to \C,\quad x  \longmapsto f_i(g^{-1}(x)).
$$ 
Since $\bar{\O}$ is $G$-invariant and is contained in the zero locus of $f_i$, 
$\bar{\O}$ is contained in $\mathcal{V}$.  
Thus $\mathcal{V}$ is a nonzero $G$-invariant subspace of $\g$ which is 
different from $\g$ (because $\mathcal{V}$ is contained in the hyperplane $\mathcal{H}_{1_{G}}$), 
whence the contradiction since $\g$ is simple.   

The Casimir element, $x \mapsto \langle x,x\rangle$ with $\langle \,,\rangle$ 
the Killing form of $\g$, vanishes on the nilpotent cone of $\g$. Hence it is 
contained in ${\mathcal I}_{\bar{\O}}$. Since it has degree 2, the minimal 
degree of the $f_i$'s is exactly 2. 
\end{proof}

To determine the reducibility of $\J_m(\bar{\O})$ for $\O$ a (nonzero) nilpotent orbit of $\g$, 
we introduce the two sufficient conditions below. 

\begin{defi} \label{d:RC}
Let $\O$ be a nilpotent orbit of $\g$. 
\begin{itemize} 
\item[1)] We say that {\em $\O$ verifies \RC$_1$} if $\pi_{\bar{\O},1}^{-1}(0)$ 
is not contained in the closure of $\pi_{\bar{\O},1}^{-1}(\O)$.   
\item[2)] Let $m\in\N^*$. 
We say that {\em $\O$ verifies \RC$_2(m)$} if for some nilpotent orbit $\O'$ 
contained in $\bar{\O}\setminus \O$, we have $\dim \pi_{\bar{\O},m}^{-1}(\O') 
\ge \dim \pi_{\bar{\O},m}^{-1}(\O) =(m+1)\dim \O$. 
\end{itemize}
\end{defi}

The following Lemma directly results from Lemma \ref{l:irr},(2). 

\begin{lemma} \label{l:RC} 
Let $\O$ be a nilpotent orbit of $\g$. 
\begin{itemize} 
\item[1)] If $\O$ verifies \RC$_1$, then $\J_1(\bar{\O})$ is reducible.    
\item[2)] If $\O$ verifies \RC$_2(m)$ for some $m\in\N^*$, 
then $\J_m(\bar{\O})$ is reducible.  
\end{itemize}
\end{lemma}

The zero nilpotent orbit verifies neither \RC$_1$ nor \RC$_2(m)$ for 
$m\in\N^*$. Since $\J_m(\mathcal{N}(\g))$ is irreducible for every $m\in\N^*$ 
(cf.~Introduction), the same goes for the regular nilpotent orbit according to 
Lemma~\ref{l:RC}. 

In view of the conditions above, let us study the zero fiber of $\pi_{\bar{\O},1}\colon 
\J_1(\bar{\O}) \to \bar{\O}$. 
As in Section~\ref{S:jet}, we identify $(\g_\O)_m$ with $(\g_\O)^{m+1}=
\underbrace{\g_\O\times \cdots\times\g_\O}_{(m+1) \text{ times}}$. 

\begin{lemma} \label{l:zero}
Let $\O$ be a nonzero nilpotent orbit of $\g$, and $m\in \N^*$. 
\begin{itemize} 
\item[1)] We have $\pi_{\bar{\O},1}^{-1}(0) \simeq \{0\}\times \g_\O$. 
In particular, $\dim \pi_{\bar{\O},1}^{-1}(0) = \dim\g_\O$. 
\item[2)] If $m\ge 2$, then 
$\dim \pi_{\bar{\O},m}^{-1}(0) \ge \dim \J_{m-2}(\bar{\O}) +\dim\g_\O 
\ge m \dim \O +{\rm codim}_{\g_\O}(\O) $. 
\end{itemize}
\end{lemma}

Part (2) of Lemma~\ref{l:zero} remains valid for an affine variety in $\C^{n}$ defined 
by homogeneous polynomials of degree at least $2$. The special case where all the 
generators have the same degree is treated in \cite[Prop.\,5.2]{Yu}.

\begin{proof} 
Clearly we may assume that $\g_\O=\g$. 
Let $f_1,\ldots,f_r$ be homogeneous generators of $\I_{\bar{\O}}$ 
that we order so that $2=d_1 \le \cdots \le d_r$, with $d_i=\deg f_i$ for any 
$i=1,\ldots ,r$ (cf.~Lemma \ref{l:deg2}). 

1) Through our identification, we can write 
$$
\pi_{\bar{\O},1}^{-1}(0) \simeq \{0\} \times \{ x \in \g \; | \;   f_i(t x ) 
= 0 \text{ mod } t^2 \, \text{ for any }\,  i=1,\ldots,r\},
$$ 
whence the statement since for any $x \in \g$ and $i\in\{1,\ldots,r\}$, we have
$f_i(t x) =t^{d_i} f_i(x)$ and $d_i \ge 2$. 

2)  Assume that $m\ge 2$. 
Let $(x_1,x_2,\ldots,x_{m-1})$ be an element of $\J_{m-2}(\bar{\O})$, 
and let $x_m \in \g$. 
Then for any $i\in\{1,\ldots,r\}$, we get 
$$
f_i(t x_1 +t^{2}x_2 +\cdots + t^{m} x_{m} ) =  f_i(t x_1  +  t^2 x_2 +\cdots +  t^{m-1} x_{m-1}) 
\text{ mod } t^{m+1}
$$ 
since $f_i$ is homogeneous of degree at least $2$.
Hence, 
$$
f_i(t x_1 +t^{2} x_2 +\cdots +  t^{m}x_{m}) =  t^{d_i} f_i(x_1 + t x_2  +\cdots + 
 t^{m-2} x_{m-1}) \text{ mod } t^{m+1}.
$$
But $f_i(x_1 + t x_2  +\cdots +  t^{m-2} x_{m-1})=0 \text{ mod } t^{m-1}$ because 
$(x_1,x_2,\ldots,x_{m-1})\in \J_{m-2}(\bar{\O})$. So,  
$$
t^{d_i} f_i(x_1 + t x_2  +\cdots + t^{m-2} x_{m-1})=0\text{ mod } t^{m+1}
$$ 
since $d_i \ge 2$. In other words, $(0,x_1,x_2,\ldots,x_m)$ is an element of $\pi_{\bar{\O},m}^{-1}(0)$.

Thus we obtain an embedding from $\J_{m-2}(\bar{\O}) \times\g$ 
into $\pi_{\bar{\O},m}^{-1}(0)$ given by 
$$
\J_{m-2}(\bar{\O}) \times\g \longrightarrow \pi_{\bar{\O},m}^{-1}(0), 
\quad ((x_1,x_2,\ldots,x_{m-1}),x_m ) \longmapsto (0,x_1,x_2,\ldots,x_{m-1},x_m ).
$$
The assertions follows. 
\end{proof}

Let $\O$ be a nonzero nilpotent orbit of $\g$, and fix $e \in \O$. 
The tangent space at $e$ to $\bar{\O}$ is the space $[e,\g]$. 
Consider the morphism 
$$
\eta_{\g,e} \colon G\times [e,\g] \longrightarrow \g,\quad (g,x) \longmapsto g(x).
$$ 

\begin{prop} \label{p:eta} 
The nonzero nilpotent orbit $\O$ verifies \RC$_1$ if and only if 
the closure of the image of 
$\eta_{\g,e}$ is strictly contained in $\g_\O$. 
\end{prop}

\begin{proof} Since $[e,\g] =[e,\g_\O]$, we may assume that $\g=\g_\O$. 
Thus, by the definition of condition \RC$_1$, we have to show that 
$\pi_{\bar{\O},1}^{-1}(0)$ is contained in $\bar{\pi_{\bar{\O},1}^{-1}(\O)}$ if 
and only if $\eta_{\g,e}$ is dominant, i.e., 
$\bar{G.[e,\g]} =\g$. 
 
By Lemma \ref{l:zero},(1), we have $\pi_{\bar{\O},1}^{-1}(0)\simeq \{0\}\times \g$. 
On the other hand, 
$$\pi_{\bar{\O},1}^{-1}(\O)= G.(\{e\}\times [e,\g]).$$
So, if $\pi_{\bar{\O},1}^{-1}(0)\subset \bar{\pi_{\bar{\O},1}^{-1}(\O)}$, 
then 
$$\{0\} \times\g \subset \bar{G.(\{e\}\times [e,\g])} \subset \bar{G.e}\times \bar{G.[e,\g]},$$
whence the inclusion $\g \subset \bar{G.[e,\g]}$, and $\eta_{\g,e}$ is dominant. 

For the other direction, observe that $\bar{\pi_{\bar{\O},1}^{-1}(\O)}$ is a closed bicone 
of $\g\times\g$ since $\O$ and $\bar{\O}$ are both subcones of $\g$. Here, by 
{\em bicone}, we mean a subset of $\g\times\g$ stable under the natural 
$(\C^*\times\C^*)$-action on $\g\times\g$. 
Therefore, if $\bar{G.[e,\g]} =\g$, then 
$$
\bar{G.(\{e\}\times [e,\g])} =
\bar{G.(\C^*e\times [e,\g])} \supset \{0\}\times \bar{G.[e,\g]} = \{0\}\times\g,
$$
whence $\pi_{\bar{\O},1}^{-1}(0) \subset \bar{\pi_{\bar{\O},1}^{-1}(\O)}$. 
\end{proof}

\begin{ex} \label{ex:2p} 
Let $p\in\N^*$ with $p\ge 2$, and $\g = \sl_{2p}(\C)$.
In the notations of Appendix~\ref{app:not}, we
claim that the nilpotent orbit $\O_{(2^{p})}$  of $\g$ associated with the 
partition $(2^p)$ verifies \RC$_1$. 
According to Proposition \ref{p:eta}, it suffices to 
prove that for the element 
$$
e:=\begin{pmatrix} 0 & I_p \\ 0 & 0
\end{pmatrix} \in \O_{(2^{p})} ,
$$
the morphism $\eta_{\g,e}$ is not dominant.  
We readily verify that $[e,\g]$ consists of matrices 
of the form
$$
\begin{pmatrix} A & C \\ 0 & -A
\end{pmatrix}
$$
with $A$ and $C$ of size $p$. In particular, $[e,\g]$ is contained in the 
closed subset $\mathcal{Z}$ of $\g$ consisting of the matrices whose characteristic 
polynomial is even. Since $\bar{G([e,\g])}$ and $\mathcal{Z}$ are both closed 
$G$-stable subsets of $\g$, we get
$$
\bar{G([e,\g])} \subset \mathcal{Z}.
$$
The diagonal matrix ${\rm diag}(1,\ldots,1,-2p+1)$ is in $\g$ 
but does not lie in $\mathcal{Z}$, for $p\ge 2$. 
Hence, $\mathcal{Z}$ is strictly contained in $\g$, and 
$\eta_{\g,e}$ is not dominant. Thus $\O_{(2^{p})}$ verifies \RC$_1$.

According to Lemma \ref{l:RC},(1), $\J_1(\bar{\O_{(2^p)}})$ is reducible. In fact, 
we can be more precise. By \cite[Thm.\,1]{We2} (see also \cite{We} or \cite[Prop.\,8.2.15]{We3}),  
the defining ideal of $\bar{\O_{(2^p)}}$ is generated by 
the entries
of the matrix $X^{2}$ as functions of $X\in\sl_{2p}(\C)$. It follows that $\J_{1}(\bar{\O_{(2^{p})}})$
can be identified with the scheme of pairs $(X_{0},X_{1}) \in \sl_{2p}(\C ) \times \sl_{2p}(\C )$ defined 
by the equations $X_{0}^{2}=0$ and $X_{0}X_{1} + X_{1}X_{0}=0$.
Using this identification, we obtain from direct computations that 
\begin{itemize}
\item $\J_1(\bar{\O_{(2^p)}})$ has exactly one irreducible component of 
dimension $4p^2=2 \dim \O_{(2^p)}$, 
\item all the other irreducible components have dimension $4p^2-1$, and 
$\pi_{\bar{\O_{(2^p)}},1}^{-1}(0)$ is one of them.
\end{itemize}
\end{ex}

\begin{rem} \label{rk:dis}
Assume that $\g=\g_\O$. A nilpotent element $e$ is distinguished if its centralizer is 
contained in the nilpotent cone. In particular, if $e$ is distinguished, then the centralizer 
of an $\sl_2$-triple $(e,h,f)$ in $\g$ is zero, and the theory of representations of $\sl_2$ 
shows that $[e,\g]$ contains $\g^h$, and hence contains a Cartan subalgebra of $\g$. 
Consequently, $G.e$ does not verify \RC$_1$.
\end{rem}

\begin{rem} \label{rk:eta}
Assume that $\g=\g_\O$. 
Since $G\times [e,\g]$ and $\g$ are irreducible varieties, $\eta_{\g,e}$ is dominant 
if and only if there is a nonempty open set $U$ consisting of 
points $a \in G\times [e,\g]$ such that $({\rm d}\eta_{\g,e})_a$ is surjective. 
The differential of $\eta_{\g,e}$ at $a=(g,[e,x])$, with $(g,x)\in G\times \g$ is given by 
$$\g\times[e,\g]  \longrightarrow \g,\quad (v,[e,w])\longmapsto  [v,[e,x]]+g([e,w]).$$  
Let us endow $G\times [e,\g]$ with the action of $G$ by left multiplication on the first factor. 
Since $\eta_{\g,e}$ is $G$-equivariant, we may assume that 
$a$ is of the form $a=(1_G,[e,x])$ with $x\in\g$. 
Then $({\rm d}\eta_{\g,e})_a$ is surjective if and only if $[\g,[e,x]]+[e,\g]=\g$. 

Consequently, $\eta_{\g, e}$ is dominant if and only if there exists $x\in \g$ such that 
$[\g,[e,x]]+[e,\g]=\g$. This allows us to affirm in some cases that $\eta_{\g,e}$ is dominant. 
For example, for $e$ in the non-distinguished nilpotent orbit $\O_{(3^2)}$ of $\sl_6(\C)$, the map
$\eta_{\g, e}$ is dominant. 
\end{rem}

\section{Little nilpotent orbits} \label{S:litt}
We introduce in this section a family of nonzero nilpotent orbits which verify 
both \RC$_1$ and \RC$_2(m)$ for every $m\in\N^*$. 
This family turns out to be useful to study the  
reducibility of jet schemes of many other orbits. 

Lemma \ref{l:zero} leads us to the following definition. 

\begin{defi}  \label{d:litt} 
Let $\O$ be a nilpotent orbit of $\g$ and let $\g_\O$ be as defined in Definition \ref{d:gO}.
We say that $\O$ is {\em little} if $0 < 2 \dim \O \le \dim \g_\O$. 
\end{defi}
In particular, neither the zero orbit nor the regular nilpotent orbit is little. 

\begin{prop} \label{p:litt}
If $\O$ is a little nilpotent orbit of $\g$, 
then $\O$ verifies \RC$_1$ and \RC$_2(m)$ for every $m \in\N^*$. 
\end{prop}

\begin{proof} Let $\O$ be a little nilpotent orbit of $\g$.  
As in the preceding proofs, we may assume that $\g=\g_\O$. 
According to Lemma~\ref{l:zero},(1), we have $\dim \pi_{\bar{\O},1}^{-1}(0)=\dim \g$. 
Since $\pi_{\bar{\O},1}^{-1}(\O)$ has dimension 
$2\dim \O \le \dim\g$,  $\O$ verifies \RC$_2(1)$ and  \RC$_1$.
Now let  $m \ge 2$. 
According to Lemma~\ref{l:zero},(2), we have 
$$
\dim  \pi_{\bar{\O},m}^{-1}(0) \ge m \dim \O +{\rm codim}_\g(\O) 
\ge(m+1)\dim \O, 
$$
since ${\rm codim}_\g(\O)  \ge \dim\O$ because $\O$ is little. 
Hence $\O$ verifies \RC$_2(m)$. 
\end{proof}

When $\mathfrak{g}$ is simple, there is a unique nonzero nilpotent orbit $\O_{\min}$, called the {\em minimal} nilpotent orbit of $\g$,
of minimal dimension and it is contained in the closure of all nonzero nilpotent orbits. 

\begin{cor} \label{c:min}
Assume that $\g$ is simple and not of type $\bf{A_1}$. 
Then $\O_{\min}$ is little. 
In particular, $\J_m(\bar{\O_{\min}})$ is reducible for every $m\in\N^*$.  
\end{cor}

\begin{proof} 
Let $e\in\O_{\min}$ that we embed into an $\sl_2$-triple 
$(e,h,f)$ of $\g$, and consider the corresponding Dynkin grading, 
$$
\g= \bigoplus_{i\in\Z} \g(i) \quad \text{ with }\quad\g(i):=\{x\in\g \; ;\; [h,x]=ix\}.
$$ 
By \cite[Lem.\,4.1.3]{CMa}, $\dim \O= \dim\g-\dim\g(0)
-\dim\g(1)$. In addition, since $e\in\O_{\min}$, we have $\dim \g(2)=1$ 
and $\g=\sum\limits_{-2\le i\le 2} \g(i)$, \cite[Prop.\,34.4.1]{TY}. 
As a result, we obtain that 
$$
\dim \g - 2\dim \O = \dim\g(0) - 2.
$$ 
The Levi subalgebra $\g(0)$ contains a Cartan subalgebra which has 
dimension at least two by our hypothesis. Hence, 
$\dim \g - 2\dim \O\ge 0$, and so $\O_{\min}$ is little.   
\end{proof}

For classical simple Lie algebras, there are explicit formulas (see Appendix~\ref{app:not}) 
for the dimension of nilpotent orbits. This allows to obtain readily examples of little nilpotent
orbits.

\begin{ex} \label{ex:litt} 
Let $n\in \mathbb{N}^{*}$ and $p,q\in\N$. 
\begin{itemize}
\item[\rm (i)\;\,] A nilpotent orbit of $\sl_n(\C )$ corresponding to a rectangular partition 
is never little. 
\item[\rm (ii)\;] The nilpotent orbit $\O_{(2^p,1^q)}$ of $\sl_{2p+q}(\C)$ is little if and only if $p,q\in\N^{*}$. 
\item[\rm (iii)] The nilpotent orbit $\O_{(p,1^q)}$ of $\sl_{p+q}(\C)$ 
is little for $q \gg p$. 
\end{itemize}
Explicit computations suggest that it is unlikely that there is a nice description 
of little nilpotent orbits in terms of partitions.
\end{ex} 

We refer to Appendix~\ref{app:not} for the notations $\P_\varepsilon(n)$, 
$\varepsilon\in\{0,1\}$, and $\O_{\bs{\lambda}}$   
with $\bs{\lambda}\in \P_\varepsilon(n)$, $n\in\N^*$.  

\begin{ex} \label{ex2:litt}
Let $\bs{\lambda}=(2^p,1^q)$, with $p\in \N^*$ and $q\in\N$.
\begin{itemize}
\item[\rm (i)\;\,] If $p$ is even, then $\bs{\lambda} \in \P_1(n)$, and 
the nilpotent orbit $\O_{\bs{\lambda}}$  of $\so_{2p+q}(\C)$ is little. 
\item[(ii)\;] If $q$ is even,  then $\bs{\lambda} \in \P_{-1}(n)$, and 
the nilpotent orbit $\O_{\bs{\lambda}}$  of $\sp_{2p+q}(\C)$ is little if and only if $p\leqslant q(q+1)/2$.  
\end{itemize}
\end{ex}

The next proposition will allow us to produce new examples 
of nilpotent orbits which verify \RC$_1$ by the "restriction" of 
certain little nilpotent orbits to Levi subalgebras.

Recall that for $\O$ a nilpotent orbit of some reductive Lie algebra $\a$, 
the semisimple Lie algebra $\a_{\O}$ was defined in Definition \ref{d:gO}. 

\begin{prop} \label{p:res}
Assume that $\g$ is simple. Let $\l$ be a 
Levi subalgebra of $\g$ with center $\z(\l)$,  
and denote by $A$ the connected subgroup of $G$ whose Lie algebra is $\a:=[\l,\l]$. 
Let $e$ be a nilpotent element of $\a$ and suppose that the following conditions 
are satisfied: 
\begin{itemize}
\item[{\rm (i)}\;\,] $\a$ contains a regular semisimple element of $\g$, 
\item[{\rm (ii)}\;] $\a_{A.e}=\a$,  
\item[{\rm (iii)}]  $2\dim G.e \le \dim \g-\dim\z(\l)$. 
\end{itemize}
Then $A.e$ verifies \RC$_1$. 
\end{prop}

\begin{proof} 
Define the following maps 
$$\theta \colon G\times \a \to \g, \; (g,x) \mapsto g(x), 
\qquad \eta = \eta_{\g,e}\colon G\times [e,\g] \to \g, \; (g,x) \mapsto g(x).$$
Observe that the image of each of the above maps is irreducible. 
Moreover, for any $x\in\g$, the map $g\mapsto (g^{-1},g(x))$ defines a 
bijection between $G_\theta(x):=\{g \in G \; ; \; g(x) \in \a\}$  
and $\theta^{-1}(\{x\})$. Similarly, we have a bijection between 
$G_\eta(x):=\{g \in G \; ; \; g(x) \in[e,\g]\}$ 
and $\eta^{-1}(\{x\})$. These bijections are isomorphisms of varieties. 

\smallskip

\noindent 
{\bf Step 1.} We shall first compute the dimension of the image of $\theta$. 

Let $L$ be the connected subgroup of $G$ whose Lie algebra is $\l$. 
By condition (i), $\a$ contains regular semisimple elements of $\g$. 
If $s$ is such an element, then $\g^s$ is a Cartan subalgebra of $\l$. Let 
$g\in G_\theta(s)$. Then $g(s) \in\a$ and $\g^{g(s)}=g(\g^s)$ is another 
Cartan subalgebra of $\l$. It follows that there exists $\tau \in L$ such that 
$\tau g \in N_G(\g^s)$, with $N_G(\g^s)$ the normalizer of $\g^s$ in $G$. 
Hence, $g \in LN_G(\g^s)$. Thus, we have obtained the inclusion 
$G_\theta(s)\subset L N_G(\g^s)$. 
On the other hand, since $L$ normalizes $\a$, we get $L \subset G_\theta(s)$ 
and therefore $\dim L\le \dim G_{\theta}(s)$. 

Let $C_G(\g^s)$ and $C_L(\g^s)$ be the centralizers of $\g^s$ in $G$ and 
$L$ respectively. 
Since $\g^s$ is a Cartan subalgebra, $C_G(\g^s)$ is connected and so, 
$C_G(\g^s)=C_L(\g^s)$ is contained in $L$. It follows that $L N_G(\g^s)$ is 
a finite union of right $L$-cosets. We deduce that 
$$
\dim \theta^{-1}(\{s\})=\dim G_\theta(s)=\dim L =\dim \a +\z(\l).
$$ 
Since the set of regular semisimple elements in $\g$ is open 
and dense, we obtain that for $s$ as above, 
$$
\dim \bar{\Im \theta}=\dim\g+\dim\a - \dim \theta^{-1}(\{s\} )= \dim\g -\dim\z(\l).
$$ 

\smallskip

\noindent 
{\bf Step 2.} We now consider the image of $\eta$. 

Let $(e,h,f)$ be an $\sl_2$-triple of $\g$. We easily check that 
$\mathfrak{c}:=\C h\oplus \g^{e}$ is a Lie subalgebra, and that $\mathfrak{c}$ stabilizes $[e,\g]$. 
Let $C$ be the connected subgroup of $G$ whose Lie algebra is $\mathfrak{c}$. 
Then $C$ is contained  in $G_\eta(x)$ for any $x\in[e,\g]$. In particular, $\dim G_\eta(x) 
\ge \dim C = 1+\dim  \g^{e}$ for $x\in[e,\g]$, and so 
$$
\dim \bar{\Im \eta} \le \dim\g +\dim\,[e,\g]-1-\dim\g^{e}=2\dim G.e -1.
$$

\smallskip

\noindent 
{\bf Step 3.} By condition (iii) and Steps 1 and 2, we deduce that $\dim 
\bar{\Im \theta} > \dim \bar{\Im \eta}$. Thus $\bar{\Im \theta}\not\subset 
\bar{\Im \eta}$. We claim that this implies that $A.e$ is \RC$_1$. 
Let us suppose on the contrary that $A.e$ is not \RC$_1$. By 
condition (ii) and Lemma \ref{l:zero},(1), $\pi_{\bar{A.e},1}^{-1}(0)=\{0\}\times \a$. 
So, $\pi_{\bar{A.e},1}^{-1}(0)$ is contained in $\bar{\pi_{\bar{A.e},1}^{-1}(A.e)}$. 
Recall from the end of Section \ref{S:jet} the notations $G_{1}$ and $A_{1}$ for $\J_{1}(G)$ and $\J_1(A)$ respectively. 
It follows that 
$$
\{0\}\times G.\a \subset G_1.(\{0\}\times \a) \subset G_1 \bar{A_1.e} \subset \bar{G_1.e},
$$
whence 
$$
\{0\} \times \bar{G.\a} \subset \bar{G_1.e}.
$$
Since $\bar{\pi_{\bar{G.e},1}^{-1}(G.e)}=\bar{G_1.e}$ (cf.~Lemma \ref{l:gpe}), 
it follows from the proof of Proposition \ref{p:eta} that 
$$
\bar{G_1.e}\cap (\{0\}\times\g) = \bar{\pi_{\bar{G.e},1}^{-1}(G.e)} 
\cap (\{0\}\times\g)=\{0\}\times \bar{G.[e,\g]}. 
$$  
Hence we get $\bar{\Im \theta} \subset 
\bar{\Im \eta}$ and the contradiction. 
\end{proof}

Suppose that $\g$ is simple. Let us fix a Cartan subalgebra $\h$ of $\g$.
Denote by $\Delta$ the root system relative to $(\g,\h)$ and let us fix a system
of simple roots $\Pi$. Given $S\subset \Pi$, we denote $\Delta_{S} = \mathbb{Z}S\cap \Delta$
the subroot system generated by $S$, and 
$$
\l_{S}= \mathfrak{h} \oplus \bigoplus_{\alpha \in  \Delta_{S}} \mathfrak{g}_{\alpha}
$$
where $\mathfrak{g}_{\alpha}$ denotes the root subspace relative to $\alpha$.
Then $\l_{S}$ is a Levi subalgebra of $\g$ and any Levi subalgebra of $\g$ is conjugate to one in this form.

Given $S\subset \Pi$, denote $\mathfrak{t} = [\l_{S} , \l_{S}]\cap \h$. Then,
$\l_{S}$ verifies condition (i) if and only if  $\mathfrak{t} \not\subset \cup_{\alpha\in\Delta} \ker \alpha$. 
To check the latter condition, it is enough to verify that for every 
$\alpha \in\Delta$, there is $\beta \in S$ 
such that $\langle \beta^\vee, \alpha \rangle \not=0$.

Thus not all Levi subalgebras of $\g$ verify condition (i) of 
Proposition \ref{p:res}. For example, if $\g$ is simple of type $\bf{B}_{\bs{\ell}}$, 
then a (maximal) Levi subalgebra whose semisimple part is simple of type 
$\bf{B}_{\bs{\ell-1}}$ does not verify the condition. The same goes for a 
Levi subalgebra  in type $\bf{C}_{\bs{\ell}}$ whose semisimple part is 
simple of type $\bf{C}_{\bs{\ell-1}}$. 

However, if $\g$ is simple of type $\bf{D}_{\bs{\ell}}$ and if $\l$ is a Levi subalgebra 
whose semisimple part is simple of type $\bf{D}_{\bs{\ell-1}}$, then $\l$ verifies 
the condition (i). Likewise, if $\g$ is simple of type $\bf{E_7}$ and if 
$\l$ is a Levi subalgebra whose semisimple part is simple of type $\bf{E_6}$, 
then $\l$ verifies the condition (i). 
Applying Proposition~\ref{p:res}, we obtain examples of nilpotent orbits 
in types $\bf{D}$ or $\bf{E_6}$ which verify \RC$_1$ that are not little. 

We list in Table \ref{tab:res} some nilpotent orbits that we obtain 
in this way. In all the examples presented in the table, the center of the Levi subalgebra is $1$-dimensional, 
and $\a$ is simple. The first and second columns give the type of the simple Lie algebras 
$\g$ and $\a$. Condition (ii) is verified in view of the discussion above.
We describe the nilpotent orbits $G.e$ and $A.e$ in the third and fourth columns respectively. 
The description for an orbit in $\g$ of type ${\bf D}$ is given in terms of partitions 
(cf.~Appendix~\ref{app:not}), while for an orbit in $\g$ of type $\bf{E_6}$ or $\bf{E_7}$, it is given by its
Bala-Carter label. 

\begin{table}[h]
$$
\begin{array}{c|c|c|c}
\mathfrak{g} & \a & G.e & A.e \\
\hline\hline
\bf{D_{6}} & \bf{D_{5}} & (3,2^{2},1^{5}) & (3,2^{2},1^{3}) \\
\hline
\bf{D_{7}} & \bf{D_{6}} & (3^{2}, 1^{8}) & (3^{2},1^{6}) \\
\hline
\bf{D_{9}} & \bf{D_{8}} & (3^{2},2^{2},1^{8}) & (3^{2},2^{2},1^{6}) \\
\hline
\bf{D_{10}} & \bf{D_{9}} & (3^{3}, 1^{11}) & (3^{3}, 1^{9}) \\
\hline
\bf{D_{10}} &\bf{D_{9}} & (4^{2}, 1^{12}) & (4^{2}, 1^{10}) \\
\hline
\bf{D_{10}} & \bf{D_{9}} & (5,2^{2}, 1^{11}) & (5,2^{2}, 1^{9}) \\
\hline
\bf{D_{10}} & \bf{D_{9}} & (5,3, 1^{12}) & (5,3, 1^{10}) \\
\hline
\bf{E_{7}} & \bf{E_{6}} & (3A_{1})' & 3A_{1} \\
\hline
\bf{E_{7}} & \bf{E_{6}} & A_{2} & A_{2} \\
\end{array}
$$
\caption{Examples of non-little nilpotent orbits satisfying RC$_1$ obtained 
by restriction.} \label{tab:res}
\end{table}

\begin{rem} ~
\begin{itemize}
\item[1)] The first (and also the last) line of Table \ref{tab:res} provides an example of a 
{\em rigid}\footnote{See Section \ref{S:ind} for the notion of rigid nilpotent orbit, and Appendices~\ref{app:not} and~\ref{app:exc}
for the description of rigid nilpotent orbits in simple Lie algebras.} nilpotent orbit which verifies \RC$_1$ and which is not little.
\item[2)] Propositions \ref{p:eta}, \ref{p:litt} and \ref{p:res}, together with Remark \ref{rk:eta},
allow us to classify all nilpotent orbits verifying \RC$_1$ in simple Lie algebras
of exceptional type. They are listed in Appendix~\ref{app:exc}.
\end{itemize}
\end{rem}


\section{Induced nilpotent orbits} \label{S:ind}
Let $\l$ be a proper Levi subalgebra of $\g$, and let $\p$ be a parabolic 
subalgebra of $\g$ with Levi decomposition $\p=\l \oplus\n$ so that
$\n$ is the nilpotent radical of $\p$. 
Let $P$, $L$ and $U$ be the connected
closed subgroups of $G$ whose Lie algebra are $\p$, $\l$ and $\n$ respectively. 
Then $P=LU$. 

The following definitions and results on induced nilpotent orbits are mostly extracted
from \cite{Ri} and \cite{LS}. We refer to~\cite[Chap. 7]{CMa} for a recent survey.

\begin{thm} \label{t:CM}
Let $\O_{\l}$ be a nilpotent orbit of $\l$. There exists a unique nilpotent
orbit ${\O}_{\g}$ in ${\g}$ whose intersection with $\O_{\l}+ \n$ is a dense 
open subset of $\O_{\l}+ \n$. Moreover, the intersection of ${\O}_{\g}$ with
$\O_{\l}+\n$ consists of a single $P$-orbit and 
${\rm codim}_{\g}({\O}_{\g})={\rm codim} _{\l}({\O}_{\l})$.
\end{thm}

The nilpotent orbit ${\O}_{\g}$ only depends on ${\l}$, and not on the choice of a parabolic
subalgebra ${\p}$ containing it. The nilpotent orbit
${\O}_{\g}$ is called the \emph{induced nilpotent orbit of $\g$ from ${\O}_{\l}$}, and it is 
denoted by Ind$_{\l}^{\g}(\O_{\l})$. 
A nilpotent orbit which is not induced in a proper way from another one
is called {\em rigid}. In type $\bf{A}$, only the zero orbit is rigid. 

\begin{rem} \label{rk:ind} ~

\begin{itemize} 
\item[1)] Let $\s_1,\ldots,\s_n$ be the simple factors of $[\g,\g]$ 
and denote by $\z(\g)$ the center of $\g$. 
Then there are Levi subalgebras $\r_1,\ldots,\r_n$ of $\s_1,\ldots,\s_n$ 
respectively such that 
$$
\l = \z(\g) \times \r_1 \times \cdots \times \r_n.
$$  
If $\O_\l$ is a nilpotent orbit of $\l$, then $\O_\l=\O_{\r_1} \times \cdots \times \O_{\r_n}$, 
where $\O_{\r_1}, \ldots , \O_{\r_n}$ are   
nilpotent orbits in the semisimple parts 
of $\r_1,\ldots,\r_n$ respectively. 
Then
$$
{\rm Ind}_{\l}^{\g}(\O_\l)= 
{\rm Ind}_{\r_1}^{\s_1}(\O_{\r_1}) \times \cdots \times {\rm Ind}_{\r_n}^{\s_n}(\O_{\r_n})= 
{\rm Ind}_{[\g,\g]\cap \l}^{[\g,\g]}(\O_\l).
$$
\item[2)] The induction property is transitive in the following sense, 
\cite[Prop.\,7.1.4]{CMa}: if $\l_1$ and $\l_2$ are two Levi subalgebras of 
$\g$ with $\l_1\subset \l_2$, then 
$$
{\rm Ind}_{\l_2}^{\g}({\rm Ind}_{\l_1}^{\l_2}(\O_{\l_1})) = {\rm Ind}_{\l_1}^{\g}(\O_{\l_1}).
$$
\item[3)] 
If $\Omega_\l$ is an $L$-orbit in $\bar{\O_\l}\setminus \O_\l$, then 
the induced nilpotent orbit of $\g$ from $\Omega_\l$ is 
contained in $\bar{\O_\g}\setminus \O_\g$. 
\end{itemize}
\end{rem}

Let $\O_\l$ be a nilpotent orbit of $\l$ and denote by $\O_\g$ the induced 
nilpotent orbit of $\g$ from $\O_\l$. According to Theorem \ref{t:CM}, 
$\O_\g \cap (\O_\l+\n)$ is a single $P$-orbit that we shall denote 
by $\O_\p$, that is 
$$
\O_\p:= \O_\g \cap (\O_\l+\n).
$$ 

\begin{lemma} \label{l:ind}
We have:
$$
\bar{\O_\p} = \bar{\O_\l} +\n, \qquad 
\bar{\O_\p}\cap \O_\g = \O_\p \qquad \text{ and }\qquad 
\bar{\O_\g}=G.(\bar{\O_\l}+\n).
$$  
\end{lemma}

\begin{proof} 
The first equality is obvious since $\O_\p$ is dense in $\O_\l + \n$ 
by definition.

Next, the inclusion $\O_\p \subset \bar{\O_\p}\cap\O_\g$ is clear. 
To show the other inclusion, assume that there is $x\in \bar{\O_\p}\cap\O_\g$, 
with $x\not\in \O_\p$. A contradiction is expected. 
Since $x \in \bar{\O_\p}\setminus \O_\p$, $\dim P.x < \dim P.e$. Hence, 
$$
\dim \g^{x} \ge \dim \p^{x} > \dim \p^{e} =\dim\g^{e}.
$$
As a consequence, $x$ is not in $\O_\g$, whence the contradiction. 

A proof of the last equality can be found in \cite[Thm.\,7.1.3]{CMa}. 
\end{proof}

For jet schemes, we have the following generalization.  

\begin{lemma} \label{l2:ind} 
We have


\begin{itemize} 
\item[1)] $\bar{\J_m(\O_\p)} = \bar{\J_m(\O_\l)} +\n_m$, 


\item[2)] $\bar{\J_m(\O_\p)} \cap \J_m(\O_\g) = \J_m(\O_\p) =
(\J_m(\bar{\O_\l}) + \n_m)\cap \J_m(\O_\g)$, 


\item[3)] $\bar{\J_m(\O_\g)}$ 
is the closure of ${G_m.\bar{\J_m(\O_\p)}}$. 
\end{itemize}  
\end{lemma}

\begin{proof}


1) Since $\O_\p \subset \O_\l +\n$, we get 
$\J_m(\O_\p) \subset  \bar{\J_m(\O_\l)} +\n_m$ because $ \bar{\J_m(\O_\l)} +\n_m$ is closed. 
Let $e'\in \O_\l$ and $x \in\n$ be such that 
$e:=e'+x$ is in $\O_\p$. From the above inclusion, we deduce that  
$$
\dim \p -\dim \p^{e} \le \dim \l -\dim \l^{e'} +\dim \n = \dim\p-\dim\g^{e},
$$
because $\dim\l^{e'}=\dim\g^{e}$ by Theorem \ref{t:CM}. 
Since $\dim \p^{e} \le\dim \g^{e}$, we get $\p^{e}= \g^{e}$, 
whence $\dim\bar{\J_m(\O_\p)}=\dim(\J_m(\O_{\l})+\n_m)$ by Lemma \ref{l:gpe} and Proposition \ref{p:smo}. 
So $\bar{\J_m(\O_\p)}$ and $\bar{\J_m(\O_\l)}+\n_m$ are irreducible varieties of the same dimension, 
and the equality follows. 

2) Taking into account Lemma \ref{l:gpe} and Proposition \ref{p:smo}, the result follows from the
same arguments as in the proof of Lemma 5.3, second equality.

3) By Lemma \ref{l:gpe}, we have 
$$
\J_m(\O_\g)=G_m.\J_m(\O_\p) \subset G_m.\bar{\J_m(\O_\p)}. 
$$
As a result, $\bar{\J_m(\O_\g)}$ is contained in the closure of 
$G_m.\bar{\J_m(\O_\p)}$. On the other hand, since $\bar{\J_m(\O_\g)}$ is 
$G_m$-stable, we get
$$
G_m.\bar{\J_m(\O_\p)}\subset \bar{\J_m(\O_\g)}.
$$ 
So the 
closure of $G_m.\bar{\J_m(\O_\p)}$ is contained in $\bar{\J_m(\O_\g)}$, 
whence the expected equality.
\end{proof}

\begin{question} \label{q:ind} 
For $m=0$, ${G_m.\bar{\J_m(\O_\p)}}$ is closed (cf.~Lemma \ref{l:ind}) 
essentially because $G/P$ is compact.   
For $m\ge 1$, $G_m/P_m$ is a trivial fibration over $G/P$ with $m$-dimensional 
affine fiber. Can we show nevertheless that $G_m.(\bar{\J_m(\O_\l})+\n_m)$ 
is closed, in other words that 
$\bar{\J_m(\O_\g)}=G_m.(\bar{\J_m(\O_\l})+\n_m)$?
\end{question}

\begin{thm} \label{t:ind}
Let $\l$ be a Levi subalgebra of $\g$, $\O_\l$ a nilpotent orbit of $\l$ 
and $\O_\g$ the induced nilpotent orbit of $\g$ from $\O_\l$. If $\O_\l$ 
verifies \RC$_2(m)$ for some $m\in\N^*$, then $\O_\g$ also verifies \RC$_2(m)$. 
\end{thm}

The rest of the section will be devoted to the proof of Theorem \ref{t:ind}.  

\begin{defi} \label{d:max}
Let $\l$ be a Levi subalgebra of $\g$. 
We say that $\l$ is a {\em maximal Levi subalgebra of $\g$} if 
the center of $[\g,\g]\cap\l$ has dimension one. 
\end{defi}

Let us first assume that $\g$ is simple and that 
$\l$ is a maximal Levi subalgebra of $\g$.  
Thus, the center $\z(\l)$ of $\l$ has dimension one. 
Let us fix a Cartan subalgebra $\h$ in $\l$ and $\Delta$ 
the root system relative to $(\g, \h)$. There exists a simple root system $\Pi$
and a subset $\Pi' \subseteq \Pi$ verifying $\card( \Pi \setminus \Pi')=1$ such that 
$\l$ is the sum of $\h$ and all the $\alpha$-root spaces for $\alpha$ in the root subsystem generated by $\Pi'$. 
Define $z$ to be the element in $\h$ such that  
$$
\alpha(z)=0\; \text{ if }\; \alpha \in \Pi'  \quad \text{ and } \quad 
\alpha(z)=1\; \text{ if } \; \alpha \in \Pi \setminus \Pi'.
$$  
Then $z$ is a generator of $\z(\l)$ and all the eigenvalues 
of $\ad z$ are integers. 

Let $m\in\N$.  Then $\ad z$ induces a $\Z$-grading on $\g_m$, 
$$
\g_m = \bigoplus_{k\in \Z} \g_m(k)\qquad \text{ with }\qquad  
\g_m(k):=\{y \in \g_m \; |\; [z,y] = k y\}.
$$
Set 
$$
\p= \bigoplus_{k \ge 0} \g_0(k) 
\qquad \text{ and } \qquad
\n =\bigoplus_{k >0} \g_0(k).
$$
Then $\p$ is a parabolic subalgebra of $\g$ where $\l = \g_{0}(0)$ is a Levi factor,
and whose nilpotent radical is $\n$.
Denote by $P,L$ and $U$ the connected closed subgroups of $G$ whose Lie algebra is $\p$, 
$\l$ and $\n$ respectively.

Observe that
$$
\l_m = \z(\l)_m\oplus[\l_m,\l_m] =\g_m(0) , \qquad 
\p_m= \bigoplus_{k \ge 0} \g_m(k), \qquad 
\n_m =\bigoplus_{k >0} \g_m(k).
$$ 

\begin{rem} \label{rk:gr} 
Clearly, for any nonzero integer $k$, 
we have $[z,\g_m(k)]=\g_m(k)$. In particular, 
$\g_m(0)= (\g_m)^{z}= \mathfrak{n}_{\g_m}(\C z)$ where $\mathfrak{n}_{\g_m}(\C z)$ is the 
normalizer of $z$ in $\g_m$. Also, if $x \in \g_m(k)$, with $k\in \N^*$, 
then $x$ is ${\rm ad}$-nilpotent, and 
$e^{\ad x} z = z + [x,z]= z-kx.$
\end{rem}

\begin{lemma} \label{l:gr} 
Let $\lambda \in\C^*$, $x \in \g_m(0)$ and $y \in \n_m$. 
If $x$ is ${\rm ad}$-nilpotent in $\g_m$ then there exists $\tau \in U_m$ 
such that $\tau (\lambda z + x + y)=\lambda z +x$. 
\end{lemma}

\begin{proof} 
For some $p >0$, $y= y_p +t$ with $y_p \in \g_m(p)$ 
and $t \in \sum\limits_{k \ge p+1}\g_m(k)$. Since $x$ is ${\rm ad}$-nilpotent, the 
sequence $\left((\ad x)^n \g_m(p)\right)_{n\in\N}$ is decreasing and 
$(\ad x)^n \g_m(p)=\{0\}$ for $n \ge \dim \g_m(p)$. Let $q\in\N$ be such that 
$y_p \in (\ad x)^q \g_m(p).$ Then 
\begin{eqnarray*}
e^{(1/p \lambda)  \ad  y_p} (\lambda z + x +y)
&=& \lambda z + e^{(1/p \lambda)  \ad  y_p} x 
+e^{(1/p \lambda)  \ad y_p} t  \\
&=& \lambda z  + x + (1/p \lambda)  [y_p,x] +t' = \lambda z  + x + y' 
\end{eqnarray*}
with $t' \in \sum\limits_{k \ge p+1}\g_m(k)$, $y':=  (1/p \lambda)  [y_p,x] +t' $ and  
$$
(1/p \lambda)  [y_p,x] \in (\ad x)^{q+1} \g_m(p).
$$
Therefore we may start again with $y'$. After a finite number of steps, 
we come to an element in $\sum\limits_{k \ge p+1}\g_m(k)$. Then we can start 
again with $p+1$ instead of $p$ and,  
after a finite number of steps, we come to an element of the expected 
form $\lambda z+x$. 
\end{proof}

\begin{lemma} \label{l:dimZ-2}
Let $\Omega$ be an $L$-orbit contained in $\bar{\O_\l}$ 
and let $X$ be an irreducible component of $\pi_{\bar{\O_\l},m}^{-1}(\bar{\Omega})$. 
Then 
$$
\dim G_m.(\z(\l) + X +\mathfrak{u}_m)= \dim X + 2 \dim \mathfrak{u}_m +1.
$$
\end{lemma}

\begin{proof} Set 
$$
C:=  \z(\l) + X +\n_m.
$$ 
Since $\Omega$ and $\bar{\Omega}$ 
are $L$-stable, $\pi_{\bar{\O_\l},m}^{-1}(\bar{\Omega})$ is $L_m$-stable and so 
is $X$. In addition, $\z(\l)$ is $L_m$-stable too. Hence, $C$ is $P_m$-stable because
$$
P_m.C=L_m U_m.(\z(\l) + X +\n_m) =L_m.(\z(\l) + X +\n_m) \subset C.
$$
Observe also that the elements of $X$ are all ad-nilpotent.

Consider the action of $P_m$ on $G_m\times C$ given by 
$\rho.(\sigma,c) = (\sigma \rho^{-1},\rho (c))$. Denote by $\bar{(\sigma,c)}$ 
the $P_m$-orbit of $(\sigma,c)\in G_m \times C$ with respect to this action, 
and denote by $G_m \times_{P_m} C$ the corresponding quotient space. 
The natural morphism
$$
G_m \times C \to \g,\quad (\sigma,c)\mapsto  \sigma(c)
$$
factors through the quotient and we obtain a morphism 
$$
\psi \colon G_m \times_{P_m} C \to \g
$$ 
whose image is $G_m.C$. 
Since $X$ and $\n_m$ are both closed cones, 
$z={1}_{G_m}(z)$ 
lies in the image of $\psi$ and 
$$
\psi^{-1}(z)= \{\overline{(\sigma,c)} \in G_m \times_{P_m} C \; ; \; 
\sigma(c)=z \}.
$$
Let $\bar{(\sigma,c)} \in \psi^{-1}(z)$. Because $z$ is ad-semisimple, $c$ 
is also ad-semisimple. Since all elements of $X$ are ad-nilpotent, 
we deduce that $c$ does not belong to $X+\n_m$. 
Also, since $U_m \subset P_m$, we may assume by Lemma \ref{l:gr} that 
$c$ is of the form $\lambda z+x$ with $\lambda \in\C^*$ and $x\in X$.  
Since $x \in \g_m(0)=(\g_m)^z$, we deduce from the uniqueness of the 
Jordan decomposition that $c=\lambda z$. In particular, $\sigma$ is in 
the normalizer $N_G(\C z)$ of $z$ in $G$, and $c =\sigma^{-1}(z)$.  

According to Remark \ref{rk:gr}, the identity component of the centralizer 
$C_{G_m}(z)$ of $z$ in $G_m$ is contained in $P_m$ and it has finite index 
in $N_{G_m}(\C z)$. Consequently, $\psi^{-1}(z)$ is a finite set. Thus, we get that 
$\dim {G_m.C} = \dim G_m\times_{P_m} C$ 
because they are both irreducible subsets. To conclude, it suffices to observe that 
$\dim G_m - \dim P_m = \dim \n_m$ and $\dim C=1+ \dim X +\dim\n_m $ 
since $\z(\l)=\C z$. 
\end{proof}

Since $\g$ is simple, its Killing form $\langle \,,\rangle$ is non-degenerate. 
Let us denote by $\phi$ the element of $\C[\g]^G$ defined 
by 
$$\forall \, x \in\g,\qquad \phi(x)=\langle x,x\rangle .$$ 
By our choice of $z$, $\phi(z)$ is a nonzero positive integer. 
Set 
$$
\mathscr{C}:= \z(\l) + \bar{\O_\l} +\mathfrak{u}
$$  

\begin{lemma} \label{l:phi}
The nullvariety in $\mathscr{C}$ of $\phi$ is $\bar{\O_\l}+\n$. 
\end{lemma}

\begin{proof} 
First of all, $\bar{\O_\l}+\n$ is contained in the nullvariety in $\mathscr{C}$ of $\phi$. 
For the other inclusion, let $u=\lambda z + x +y$ 
be in $\mathscr{C}$, with $\lambda\in\C$, $x \in \bar{\O_\l}$ and $y\in \n$ such that 
$\phi(u)=0$. We have 
$$0 = \phi(u) = 
\langle \lambda z + x +y , \lambda z + x +y\rangle = \lambda^2 \langle z,z \rangle + 
\langle x ,  x \rangle= \lambda^2 \langle z,z\rangle
$$ 
since $\n$ is orthogonal to $\p$, $\z(\l)$ is orthogonal to $[\l,\l]\oplus\n$ 
and since $\langle x ,  x \rangle =\phi(x)=0$. 
Hence $\lambda=0$ since $\phi(z)\not=0$. So, $u$ lies in $\bar{\O_\l}+\n$, 
whence the other inclusion. 
%
\end{proof}

Let $\phi^{(0)},\ldots,\phi^{(m)}\in \C[\g_m]$ be the polynomials 
as defined in Remark \ref{rk:aff} relative to $\phi$. According to Lemma \ref{l:inv}, 
they are $G_m$-invariant. In particular, $\phi^{(0)}$ is $G_m$-invariant. 

\begin{lemma} \label{l:null}
Let $\Omega_\l$ be an $L$-orbit contained in $\bar{\O_\l}$ and 
set $\Omega_\g:={\rm Ind}_{\l}^{\g}(\Omega_\l)$. Then: 
\begin{itemize}
\item[1)] the nullvariety in 
$G_m.(\z(\l)+\pi_{\bar{\O_\l},m}^{-1}( \bar{\Omega_\l} ) +\n_m)$ of $\phi^{(0)}$ 
is contained in $\pi_{\bar{\O_\g},m}^{-1}(\bar{\Omega_\g})$,   
\item[2)] $\dim \pi_{\bar{\O_\g},m}^{-1}(\bar{\Omega_\g}) 
\ge \dim \pi_{\bar{\O_\l},m}^{-1}( \bar{\Omega_\l} ) + 2 \dim \mathfrak{u}_m$. 
\end{itemize}
\end{lemma}

\begin{proof} 
Let us denote by $Y$ the nullvariety in 
$G_m.(\z(\l)+\pi_{\bar{\O_\l},m}^{-1}( \bar{\Omega_\l} ) +\n_m)$ of $\phi^{(0)}$. 
First of all, observe that $Y$ contains $0$ because $\z(\l)$, 
$\pi_{\bar{\O_\l},m}^{-1}( \bar{\Omega_\l} )$ and $\n_m$ 
are closed cones. In particular, $Y$ is nonempty. 


1) Let $u = g.(\lambda z + x + y)$ be in $Y$, with $g\in G_m$, 
$\lambda\in\C$, $x \in \pi_{\bar{\O_\l},m}^{-1}( \bar{\Omega_\l} )$ and $y\in \n_m$ 
such that $\phi^{(0)}(u)=0$. 
Since $\phi^{(0)}$ is $G_m$-invariant, we get, setting 
$x_0 := \pi_{\bar{\O_\l},m}(x)$ and $y_0 := \pi_{\n,m}(y)$, 
$$0=\phi^{(0)}(u)= \phi^{(0)}(\lambda z + x + y) 
=\phi(\lambda z + x_0 + y_0) = \lambda^2 \phi(z)$$
by the computations of the proof of Lemma \ref{l:phi}. 
Hence $\lambda =0$ since $\phi(z)\not=0$. 
So $u$ lies in $G_m.(\pi_{\bar{\O_\l},m}^{-1}( \bar{\Omega_\l} ) +\n_m)$. 
But 
$$G_m.(\pi_{\bar{\O_\l},m}^{-1}( \bar{\Omega_\l} ) +\n_m) 
\subset G_m.(\J_m(\bar{\O_\l}) +\n_m) \subset G_m.\J_m(\bar{\O_\g}) =\J_m(\bar{\O_\g}) 
$$
because $\J_m(\bar{\O_\g})$ is $G_m$-invariant. 
Thus $Y$ is contained in $\J_m(\bar{\O_\g})$. 
Then it remains to observe that for $u \in Y$, 
$$\pi_{\bar{\O_\g},m}(u) \in  G.(\bar{\Omega_l} +\n)
=\bar{\Omega_\g}
$$
by Lemma \ref{l:ind}. 
In conclusion, $Y$ is contained in $\pi_{\bar{\O_\g},m}^{-1}(\bar{\Omega_\g})$. 

2) Let $X$ be an irreducible component of $\pi_{\bar{\O_\l},m}^{-1}(\bar{\Omega})$ 
of maximal dimension, and let $Y'$ be the nullvariety in 
$G_m.(\z(\l)+X+\n_m)$ of $\phi^{(0)}$.
The function $\phi^{(0)}$ is not identically zero on $G_m.(\z(\l)+X+\n_m)$ since $z \in G_m.(\z(\l)+X+\n_m)$ 
and $\phi^{(0)}(z)=\phi(z)\not=0$. 
Since $Y'$ is irreducible, we deduce by Lemma \ref{l:dimZ-2} and our choice of $X$ that 
$$\dim Y' = \dim G_m.(\z(\l)+X+\n_m) - 1 = \dim X + 2 \dim \n_m 
= \dim \pi_{\bar{\O_\l},m}^{-1}( \bar{\Omega_\l} ) + 2\dim \n_m,
$$
whence the statement by 1). 
\end{proof}

\begin{prop} \label{p:ind}
If for some $L$-orbit $\Omega_\l$ in $\bar{\O_\l}$, we have 
$\dim\pi_{\bar{\O_\l},m}^{-1}(\bar{\Omega_\l}) \ge \dim \pi_{\bar{\O_\l},m}^{-1}
(\O_\l)$, then 
$\dim\pi_{\bar{\O_\g},m}^{-1}(\bar{\Omega_\g}) \ge 
\dim \pi_{\bar{\O_\g},m}^{-1}(\O_\g)$, where $\Omega_\g$ is the induced nilpotent 
orbit of $\g$ from $\Omega_\l$. 
\end{prop}


\begin{proof} 
Assume that for some $L$-orbit $\Omega_\l$ in $\bar{\O_\l}$, we have 
$\dim\pi_{\bar{\O_\l},m}^{-1}(\bar{\Omega_\l}) \ge \dim \pi_{\bar{\O_\l},m}^{-1}(\O_\l)$. 
Then by Lemma \ref{l:null}, we have 
\begin{eqnarray*}
\dim \pi_{\bar{\O_\g},m}^{-1}(\bar{\Omega_\g}) 
\ge \dim \pi_{\bar{\O_\l},m}^{-1}( \bar{\Omega_\l} ) + 2 \dim \mathfrak{u}_m 
& \ge & \dim \pi_{\bar{\O_\l},m}^{-1}(\O_\l)+ 2 \dim \mathfrak{u}_m \\
& = & (m+1) \dim \O_\l + 2(m+1)\dim \n.
\end{eqnarray*} 
To conclude, it remains to observe that $\pi_{\bar{\O_\g},m}^{-1}(\O_\g)$ 
has dimension $(m+1) \dim \O_\l + 2(m+1) \dim \n$ 
because $\dim \O_\g=2\dim\n+\dim\O_\l$ from Theorem \ref{t:CM}. 
\end{proof}

\begin{rem} \label{rk3:ind}
The above proof actually shows that $\pi_{\bar{\O_\g},m}^{-1}(\bar{\Omega_\g})$ 
has dimension at least $2(m+1) \dim \n + \dim\pi_{\bar{\O_\l},m}^{-1}(\bar{\Omega_\l})$ 
even if $\Omega_\l$ does not verify the hypothesis of the proposition. 
This can be used in practice to give an estimation of 
$\dim\pi_{\bar{\O_\g},m}^{-1}(\bar{\O_\g} \setminus \O_\g)$. 
\end{rem}

We are now in a position to prove the main result of the section. 

\begin{proof}[Proof of Theorem \ref{t:ind}] 
Let $\l$ be a Levi subalgebra of $\g$. 
Then there is a finite sequence of Levi subalgebras 
$$
\l=\l_0\subset \l_1\subset \l_1 \subset \cdots \subset \l_k = \g
$$ 
such that $\l_{i-1}$ is a maximal Levi subalgebra of $\l_i$
for every $i \in \{1,\ldots,k\}$. 

Let $\O_\l$ be a nilpotent orbit of $\l=\l_0$ verifying \RC$_2(m)$ for some $m\in\N$, 
and set for $i \in \{1,\ldots,k\}$, 
$$
\O_{\l_i}={\rm Ind}_{\l_{i-1}}^{\l_i}(\O_{\l_{i-1}}).
$$ 
Since induction is transitive, cf.~Remark \ref{rk:ind},(2), we get 
$$
\O_\g:={\rm Ind}_{\l}^{\g}(\O_\l) ={\rm Ind}_{\l_{k-1}}^{\l_k}( {\rm Ind}_{\l_{k-2}}^{\l_{k-1}} 
( \ldots ( {\rm Ind}_{\l_{0}}^{\l_1}(\O_{\l_{0}})))).
$$  
So, in order to proof Theorem \ref{t:ind}, we may assume that $\l$ is maximal in $\g$. 
Let us write $\O_\l$ as a product 
$\O_\l=\O_{\r_1} \times\cdots\times \O_{\r_n}$, with the $\r_j$'s as in Remark \ref{rk:ind},(1). 
Since $\O_\l$ verifies \RC$_2(m)$, $\O_{\r_j}$ verifies 
\RC$_2(m)$ for some $j \in \{1,\ldots,n\}$. 
Since $\l$ is maximal in $\g$, either $\r_j=\s_j$ and ${\rm Ind}_{\r_j}^{\s_j}(\O_{\r_j})$ 
obviously verifies \RC$_2(m)$ too, 
or $\r_j$ is maximal in $\s_j$ and by Proposition~\ref{p:ind}, 
${\rm Ind}_{\r_j}^{\s_j}(\O_{\r_j})$ verifies \RC$_2(m)$ as well. 
Indeed, since $\O_{\r_j}$ verifies \RC$_2(m)$, for some $\Omega_{\r_j}$ 
in  $\bar{\O_{\r_j}}\setminus \O_{\r_j}$, $\dim\pi_{\bar{\O_{\r_j}},m}^{-1}(\bar{\Omega_{\r_j}}) 
\ge \dim \pi_{\bar{\O_{\r_j}},m}^{-1}(\O_{\r_j})$ and Proposition~\ref{p:ind} applies. 
In both cases, by 
Remark \ref{rk:ind},(3), we conclude that $\O_\g:={\rm Ind}_{\l}^{\g}(\O_\l)$ verifies 
\RC$_2(m)$. 
\end{proof}


\section{Consequence of Theorem \ref{t:ind}} \label{S:csq} 
Theorem \ref{t:ind} allows us to answer the reducibility problem for 
many nilpotent orbits. 

Recall from the beginning of Section \ref{S:nil} that if $\O$ is a nilpotent orbit of a reductive Lie algebra $\g$ with
simple factors $\mathfrak{s}_{1},\dots , \mathfrak{s}_{m}$,
then $\O = \O_{1}\times \cdots \times \O_{m}$ where $\O_{i}$ is a nilpotent orbit
of $\mathfrak{s}_{i}$. We shall say that {\em $\O$ has
a little factor} if  there exists $i$ such that $\O_{i}$ is a little nilpotent orbit of $\mathfrak{s}_{i}$.

The following result is a direct consequence 
of Theorem \ref{t:ind} and Proposition \ref{p:litt}. 

\begin{thm} \label{t2:ind} 
Any nilpotent orbit induced from a nilpotent orbit that has a little factor verifies \RC$_2(m)$ for every $m\in\N^*$. 
\end{thm}

When $\g$ is simple, there is a unique nilpotent orbit $\O_{\rm subreg}$ of $\g$, 
called the subregular nilpotent orbit, such that $\mathcal{N}(\g) \setminus \O_\reg=\bar{\O_{\rm subreg}}$.  
It has codimension ${\rm rk}\,\g+2$ in $\g$. 

\begin{cor} \label{c:sub}
Assume that $\g$ simple and not of type $\bf{A_1}$, 
$\bf{B_2}={\bf C_2}$ or $\bf{G_2}$. 
Then the subregular nilpotent orbit $\O_{\rm subreg}$ of $\g$ 
verifies \RC$_2(m)$ for every $m\in\N^*$. In particular, $\J_m(\bar{\O_{\rm subreg}})$ 
is reducible for every $m\in\N^*$.  
\end{cor}

\begin{proof} 
Assume first that $\g$ has type $\bf{A_2}$. Then $\g=\sl_3(\C)$ and
$\O_{\rm subreg}=\O_{\min}=\O_{(2,1)}$. 
Hence, $\O_{\rm subreg}$ is little and verifies \RC$_2(m)$ for every $m\in\N^*$ 
according to Corollary \ref{c:min}. 

Assume now that $\g$ is simple with rank $\ge 3$. 
Then there exists a Levi subalgebra $\l$ of $\g$ such that $[\l,\l]$ is simple of type $\bf{A_2}$, 
and the subregular nilpotent orbit of $\g$ is induced from that of $[\l,\l]$ for dimension 
reasons (cf.~Theorem \ref{t:CM}).  Therefore, the 
theorem follows from the case $\sl_3(\C)$ and Theorem \ref{t2:ind}. 
\end{proof}

\begin{rem} \label{rk:sub}
Outside types $\bf{A}$ and $\bf{B}$, the subregular nilpotent orbit of a simple Lie algebra is distinguished. 
Thus Corollary~\ref{c:sub} provides examples of distinguished nilpotent orbits 
which verify \RC$_2(m)$ for every $m\in\N^*$. In particular, according to 
Remark~\ref{rk:dis}, these nilpotent orbits verify \RC$_2(1)$ but 
not \RC$_1$. 
\end{rem}

\begin{rem} \label{rk2:sub}
For $\g=\sp_4(\C)\simeq \so_5(\C)$, 
we can show that $\J_1(\bar{\O_{\rm subreg}})$ is irreducible. 

Let us detail this example where the computations are explicit. 
Let $\g=\sp_4(\C)$.
The subregular nilpotent orbit is $\O_{(2^2)}$. By~Appendix~\ref{app:not}, it has dimension 6, and
its singular locus is the union of two nilpotent orbits, $\O_{(2,1^2)}=\O_{\min}$ and the zero orbit. 

Using \cite[Thm.\,1]{We2} (see also \cite{We} or \cite[Prop.\,8.2.15]{We3}) 
and the realization of $\sp_4(\C)$ as the set of anti-self-adjoint matrices for the symplectic 
form, we can show that 
the defining ideal of $\bar{\O_{(2^2)}}$ is generated by 
the entries
of the matrix $X^{2}$ as functions of $X\in\sp_{4}(\C)$\footnote{Here, we have used the 
computer program \texttt{Macaulay2} to check that these equations indeed generate a reduced ideal.}. 
It follows that $\J_{1}(\bar{\O_{(2^2)}})$ 
can be identified with the scheme of pairs $(X_{0},X_{1}) \in \sp_{4}(\C ) \times \sp_{4}(\C )$ defined 
by the equations $X_{0}^{2}=0$ and $X_{0}X_{1} + X_{1}X_{0}=0$.

Using this identification, we obtain from direct computations that 
$$\dim \pi_{\bar{\O_{(2^2)}},1}^{-1}(\O_{(2,1^2)}) =11 \qquad\text{ and }\qquad  
\dim\pi_{\bar{\O_{(2^2)}},1}^{-1}(0)=10.$$
Furthermore, there is no smooth points of $\J_1(\bar{\O_{(2^2)}})$ in 
$\pi_{\bar{\O_{(2^2)}},1}^{-1}(\O_{(2,1^2)})\, \cup \,  \pi_{\bar{\O_{(2^2)}},1}^{-1}(0)$. 
To see this, we have computed 
the dimension of the 
tangent space to $\J_1(\bar{\O_{(2^2)}})$ at generic points in 
$\pi_{\bar{\O_{(2^2)}},1}^{-1}(\O_{(2,1^2)})$ and $\pi_{\bar{\O_{(2^2)}},1}^{-1}(0)$. 
For the points in $\pi_{\bar{\O_{(2^2)}},1}^{-1}(\O_{(2,1^2)})$, the smallest dimension 
for the tangent space is 13; for the points in $\pi_{\bar{\O_{(2^2)}},1}^{-1}(0)$, 
the dimension is 14. 

Now, if $\J_1(\bar{\O_{(2^2)}})$ were reducible, it would have an irreducible component 
of dimension 10 or 11 by the above equalities. This is not possible according to 
the computations of the tangent space dimensions. 
Hence, $\J_1(\bar{\O_{(2^2)}})$ is irreducible. 
\end{rem}


\subsection*{Classical types}
We now summarize our conclusions for the case where $\mathfrak{g}$ is simple of classical type. 
We refer to Appendix~\ref{app:not} for the notations relative to the induction of nilpotent 
orbits in the classical cases.  

\begin{thm}[{Type $\bf{A}$}] \label{t:A}
Let $n\in \N^*$, $n \ge 2$, and let $\bs{\lambda}\in\P(n)$. 
Suppose that $\bs{\lambda}$ is non rectangular, then the nilpotent orbit $\O_{\bs{\lambda}}$ 
of $\mathfrak{sl}_n(\C)$ verifies \RC$_2(m)$ for every $m\in\N^*$. In particular, 
$\J_m(\bar{\O_{\bs{\lambda}}})$ is reducible for every $m\in\N^*$. 
\end{thm}

\begin{proof} 
Suppose that $\bs{\lambda}  = (\lambda_{1},\dots ,\lambda_{r}) 
\in \P(n)$ is non rectangular, with $1<r<n$. Then there exists $1\leqslant p < r$ 
such that $\lambda_{p} > \lambda_{p+1}$. It follows that
$$\bs{\lambda}  ={\rm  Ind}_{(n-p-r,p+r)}^{n} \bs{\Lambda}$$
where
$$\bs{\Lambda} = \left( ( \lambda_{1}-2, \dots ,\lambda_{p}-2 , 
\lambda_{p+1}-1, \dots ,\lambda_{r}- 1) , (2^{p} , 1^{r-p}) \right).$$
Thus any non rectangular partition of $n$ can be induced from a partition of the
form $(2^{p}, 1^{q})$ with $p,q \in \N^{*}$.
According to Example \ref{ex:litt}, $\O_{(2^p,1^q)}$ is little for $p,q \in \N^{*}$. 
Hence the theorem follows from Theorem \ref{t2:ind}. 
\end{proof}

\begin{rem} \label{rk:rect}
It is not difficult to see that rectangular partitions can only be induced from a rectangular one. So they
cannot be induced from a nilpotent orbit that has a little factor (cf.~Example~\ref{ex:litt}).

In fact, for the rectangular case, the theorem is not true. First of all, it is obvious not true for $\bs{\lambda}=(n)$ and $\bs{\lambda}=(1^n)$. 
Let us look at some special cases. 

\begin{itemize}
\item[1)] Let $\bs{\lambda}=(2^p)$ with $2p=n$. Then we saw in Example  \ref{ex:2p} that
$\O_{\bs{\lambda}}$ is \RC$_1$, and that all the irreducible 
components of $\J_1(\bar{\O_{\bs{\lambda}}})$ different from $\pi_{\bar{\O_{\bs{\lambda}}},1}^{-1}(\O_{\bs{\lambda}})$ 
has codimension one. In particular, it is not \RC$_{2}(1)$. 
\item[2)] Let $\bs{\lambda}=(3^2)$. By \cite{We2} (see also \cite{We} or \cite[Prop.\,8.2.15]{We3}), the  defining ideal of $\bar{\O_{\bs{\lambda}}}$
is generated by $\mathrm{tr}(X^{2})$ 
and the entries of the matrix $X^{3}$ as functions of $X\in \sl_{6}(\mathbb{C})$.
By Appendix~\ref{app:not}, the singular locus of $\bar{\O_{\bs{\lambda}}}$ is the 
finite union of the nilpotent orbits $\O_{\bs{\mu}}$ 
with 
$$\mu\in \{(3,2,1),(3,1^3),(2^3),(2^2,1^2),(2,1^4),(1^6)\} \subset \P(6),$$  
and the respective dimensions of $\pi_{\bar{\O_{\bs{\lambda}}},1}^{-1}(\O_{\bs{\mu}})$ 
are 47, 44, 44, 47, 44, 35. Note that $\J_1(\bar{\O_{\bs{\lambda}}})$ has dimension 48. 
Next, we obtain that the respective dimensions of the tangent 
space to $\J_1(\bar{\O_{\bs{\lambda}}})$ at generic points in 
$\pi_{\bar{\O_{\bs{\lambda}}},1}^{-1}(\O_{\bs{\mu}})$, with $\bs{\mu}$ running through the above 
set, are 49, 51, 51, 48, 52, 69. 
Arguing as in Remark \ref{rk2:sub}, 
we conclude that $\J_1(\bar{\O})$ is irreducible. 
\end{itemize} 
\end{rem}

Thereby, from Remark \ref{rk:rect},(1) and (2), we have complete answers for the 
reducibility of $\J_1(\bar{\O})$ for any nilpotent orbit $\O$ in $\sl_n(\C)$, 
for $n\le 7$, and for any nilpotent orbit $\O$ in $\sl_p(\C)$, with $p$ a prime number. 

In the other classical simple Lie algebras, we have the following result. 

\begin{thm}[Types $\bf{B}$, $\bf{C}$, $\bf{D}$]  \label{t:BCD}
Let $\bs{\lambda}=(\lambda_1,\ldots,\lambda_t) \in \P_\varepsilon(n)$ 
with $\varepsilon\in\{1,-1\}$, and set $\lambda_{t+1}=0$. 
\begin{itemize}
\item[1)] Suppose that $\varepsilon=1$ and there exist $1\le k < \ell \le t$ such that 
$\lambda_{k} \ge \lambda_{k+1} +2$ and $\lambda_{\ell} \ge \lambda_{\ell+1} 
+2$, then the nilpotent orbit $\O_{\bs{\lambda}}$ of $\so_{n}(\C)$ verifies \RC$_2(m)$ 
for every $m\in\N^*$. 
\item[2)] Suppose that $\varepsilon=-1$ and there exist $1\le k < \ell \le t$ such that 
$\lambda_{k} \ge \lambda_{k+1} +2$ and $\lambda_{\ell} \ge \lambda_{\ell+1} 
+ 2$, then the nilpotent orbit $\O_{\bs{\lambda}}$ of $\sp_{n}(\C)$ verifies \RC$_2(m)$ 
for every $m\in\N^*$. 
\item[3)] Suppose that $\varepsilon=1$ and that $\bs{\lambda}$ is very even. Then 
both $\O_{\bs{\lambda}}^{I}$ and $\O_{\bs{\lambda}}^{I\!I}$ verfiy \RC$_2(m)$ for every 
$m\in\N^*$. 
\end{itemize}
In particular, $\J_m(\overline{\O_{\bs{\lambda}}})$ is reducible for every $m\in\N^{*}$.
\end{thm}

\begin{proof} 
Let $\bs{\lambda}  = (\lambda_{1},\dots, \lambda_{t}) \in \P_{\varepsilon}(n)$, 
set $\lambda_{t+1}=0$, and suppose that there exist $1\le k  < \ell \le t$ 
such that $\lambda_{k} \ge \lambda_{k+1} +2$ and $\lambda_{\ell} \ge 
\lambda_{\ell+1}+2$ as in the theorem. Then 
$$
\bs{\lambda}  = {\rm Ind}^{n,\varepsilon}_{(\ell+k ; n-2(\ell+k))} \bs{\Gamma}
$$
where 
$$
\bs{\Gamma} :=  \left( (2^{k},1^{\ell-k}) ; 
(\lambda_{1}-4 , \dots ,\lambda_{k}-4, \lambda_{k+1}-2 ,\dots ,\lambda_{\ell}-2,
\lambda_{\ell+1}, \dots ,\lambda_{t}) \right).
$$
So $\bs{\lambda}$ is induced from a partition in $\P(n)$ 
of the form $(2^{p},1^{q})$, with $p,q\in\N^*$. By Example \ref{ex:litt}, 
the partition $(2^{p},1^{q})$ is little. This concludes the proof of parts (1) and (2)
according to Theorem \ref{t2:ind}. 

Finally, if $\bs{\lambda} \in \P_1(n)$ is very even, then $\O_{\bs{\lambda}}$ is induced 
from the nilpotent orbit $\O_{(2^t)}$ of $\so_{2t}(\C)$ which is little by Example \ref{ex2:litt}. 
Again, we conclude thanks to Theorem \ref{t2:ind}. 
\end{proof}

\begin{rem} \label{rk:BCD}
Unlike the type {\bf A} case, in types $\bf{B}$, $\bf{C}$, $\bf{D}$, orbits other than the ones considered 
in Theorem~\ref{t:BCD} can be induced from little ones. For example, for 
$\lambda, p, q \in \mathbb{N}^{*}$ with $p$ even, we have
$\bs{\lambda} = \bigl( (2\lambda)^{p} , (2\lambda-1)^{q} \bigr)  \in \P_1 \bigl( 2\lambda (p+q)-q \bigr)$ and
$\bs{\lambda}$ does not verify the conditions of Theorem \ref{t:BCD}. However, we have
$$
\bs{\lambda} = \bigl( (2\lambda)^{p} , (2\lambda-1)^{q} \bigr) 
= {\rm Ind}^{2\lambda (p+q)-q,1}_{( (\lambda-1)(p+q) ; 2p+q )} \bigl(  (\lambda-1)^{p+q} ; (2^{p},1^{q})\bigr) 
$$
Since the nilpotent orbit of $\so_{2p+q}(\C )$ corresponding to the partition $(2^{p},1^{q})$ is little 
(cf.~Example~\ref{ex2:litt}), $\O_{\bs{\lambda}}$ verifies \RC$_2(m)$ for all 
$m\in\N^*$. 

Unfortunately, in types $\bf{B}$, $\bf{C}$, $\bf{D}$, we have not found a nice exhaustive description 
of nilpotent orbits that can be reached by induction from a little nilpotent 
orbit. Computations using \texttt{GAP4} show that a big proportion of partitions 
can be induced from little ones. See Appendix~\ref{app:stat} for some numerical data. 
\end{rem}

\subsection*{Exceptional types} 
Our conclusions for the exceptional types are summarized in Appendix~\ref{app:exc}. 
More precisely, we can find in Appendix~\ref{app:exc} the list of nilpotent orbits in a 
simple Lie algebra of exceptional type which can be induced from a little one.


\section{Applications, remarks and comments} \label{S:app}
We give in this section applications to geometrical properties of 
nilpotent orbit closures. 

\subsection*{Nilpotent orbits closures and complete intersections} 
Let $\O$ be a nilpotent orbit of the reductive Lie algebra $\g$. 

\begin{thm} \label{t:Nam}
If $\O$ verifies \RC$_1$ or \RC$_2(m)$ for some $m\ge 1$, then 
$\bar{\O}$ is not a complete intersection. 
\end{thm}

\begin{proof} 
Since the singular locus of $\bar{\O}$ is $\bar{\O}\setminus \O$ (cf.~Introduction), 
it has codimension at least two in $\bar{\O}$. 
Hence, $\bar{\O}$ is normal if it is a complete intersection. 
If so, by \cite{Hi} or \cite{Pa}, it has rational singularities. The theorem
is then of direct consequence of Theorem \ref{t:Mus}. 
\end{proof}

In the papers of Namikawa, \cite{Na}, and Brion-Fu, \cite{BF},  
the authors use symplectic resolutions of singularities of nilpotent orbit 
closures to prove the above corollary for arbitrary nilpotent orbits in $\g$.  
The foregoing provides an alternative method to obtain that result through 
jet schemes in a large number of cases (see Section \ref{S:csq}).  
There are other approaches in the jet scheme setting to show that 
$\bar{\O}$ is not a complete intersection. Let us give an example.

\begin{ex} \label{ex:Nam}
The computations described in Remark \ref{rk:rect},(2), show that 
for generic $x \in \pi_{\bar{\O_{(3^2)}},1}^{-1}(\O_{(2^2,1^2)})$, the tangent 
space at $x$ of $\J_1(\bar{\O_{(3^2)}})$ has dimension $48= \dim \J_1(\bar{\O_{(3^2)}})$. 
Hence, such an $x$ is a smooth point of $\J_1(\bar{\O_{(3^2)}})$, because $\J_1(\bar{\O_{(3^2)}})$ 
is irreducible, which does not belong to $\pi_{\bar{\O_{(3^2)}},1}^{-1}(\O_{(3^2)})$. So, 
$(\J_1(\bar{\O_{(3^2)}}))_\reg \not= \pi_{\bar{\O_{(3^2)}},1}^{-1}(\O_{(3^2)})$ and 
by Theorem \ref{t:Mus},(3), $\bar{\O_{(3^2)}}$ is not a complete intersection. 

Unfortunately, theses arguments cannot be used for 
the nilpotent orbit $\O_{(2^2)}$ of $\sp_4(\C)$ because, in this case,
the computations of Remark \ref{rk2:sub} show that we exactly have 
$(\J_1(\bar{\O_{(2^2)}}))_\reg = \pi_{\bar{\O_{(2^2)}},1}^{-1}(\O_{(2^2)})$. 
\end{ex}


\subsection*{Examples and counter-examples} 
Our results provide many examples showing that the converse 
of Proposition \ref{p:geo} for irreducibility is not true. 
Since the nilpotent cone $\mathcal{N}(\g)$ is normal, the following 
result illustrates that the converse of Proposition \ref{p:geo} 
 for normality is also not true.  

\begin{prop} \label{p:norm} 
Assume that $\g$ simple, and let $m\in\N$. 
Then $\J_m(\mathcal{N}(\g))$ is normal if and only if $m=0$. 
\end{prop}

\begin{proof}  
Since $\J_0(\mathcal{N}(\g)) \simeq \mathcal{N}(\g)$ is normal, 
we have to show that for any $m \in\N^*$, $\J_m(\mathcal{N}(\g))$ is not normal. 

Fix $m\in\N^*$. 
Let $\ell$ be the rank of $\g$, and let $p_1,\ldots,p_\ell$ be homogeneous 
generators of $\C[\g]^G$ so that 
$$
\mathcal{N}(\g)=\Spec \C[\g]/(p_1,\ldots,p_\ell).
$$
By Remark \ref{rk:aff}, we get
$$
\J_m(\mathcal{N}(\g))\simeq \Spec \C[\g_m]/(p_i^{(j)};\, i=1,\ldots,\ell,\,j=0,\ldots,m). 
$$
Since $\mathcal{N}(\g)$ is a complete intersection with rational singularities, 
$\J_m(\mathcal{N}(\g))$ is irreducible and reduced by Theorem \ref{t:Mus}. 
So, it is generically reduced and we have 
\begin{eqnarray} \label{eq:norm}
(\J_m(\mathcal{N}(\g)))_{\rm reg} &=&  
\{x = x_0 + x_1 t +\cdots x_m t^m \in \J_m(\mathcal{N}(\g)) \; | \; \d p_i^{(j)}(x_0,x_1,\ldots,x_m) 
\\\nonumber
&& \qquad\quad \text{ are linearly independent for } i=1,\ldots,\ell 
\text{ and }j=0,\ldots,m \}.
\end{eqnarray}
According to \cite[Lem.\,3.3,(i)]{RT}, the vectors $\d p_i^{(j)}(x_0,x_1,\ldots,x_m)$, 
for $i\in\{1,\ldots,\ell\}$, $j\in\{0,\ldots,m\}$ and $x_0 + x_1 t +\cdots x_m t^m \in \g_m$,  
are linearly independent if and only if the vectors 
$\d p_1(x_0),\ldots,\d p_\ell(x_0)$ are linearly independent. 
But by \cite{Ko}, the later condition is satisfied if and only if $x_0$ is a regular 
element of $\g$. 
Therefore by (\ref{eq:norm}), we get 
\begin{eqnarray} \label{eq2:norm}
(\J_m(\mathcal{N}(\g)))_{\rm reg} = \pi_{\mathcal{N}(\g),m}^{-1}(\O_{\reg}) 
& \text{ and } & (\J_m(\mathcal{N}(\g)))_{\rm sing} = \pi_{\mathcal{N}(\g),m}^{-1}(\overline{\O_{\rm subreg}}) 
\end{eqnarray}
since $\mathcal{N}(\g)\setminus \O_\reg =\bar{\O_{\rm subreg}}$.
Then by Serre's criterion, it is enough to show that
$\pi_{\mathcal{N}(\g),m}^{-1}(\overline{\O_{\rm subreg}})$ has codimension one in 
$\J_m(\mathcal{N}(\g))$, or else that 
\begin{eqnarray} \label{eq3:norm}
\dim \pi_{\mathcal{N}(\g),m}^{-1}(\overline{\O_{\rm subreg}}) 
\ge \dim \J_m(\mathcal{N}(\g)) -1.
\end{eqnarray}
The zero orbit of $\sl_2(\C)$ has codimension 2 in $\mathcal{N}(\sl_2(\C))$. 
Hence, for dimension reasons, $\O_{\rm subreg}$ is the induced nilpotent orbit 
from $0$ in any Levi subalgebra $\l$ of $\g$ with semisimple part $[\l,\l]$ isomorphic to $\sl_2(\C)$. 
So by Remark \ref{rk3:ind}, 
in order to prove (\ref{eq3:norm}), it suffices to show the statement for $\g=\sl_2(\C)$. 

If $\g=\sl_2(\C)$, then $\overline{\O_{\rm subreg}}=0$ but by Lemma \ref{l:zero},(2),  
$$\dim \pi_{\mathcal{N}(\sl_2(\C)),m}^{-1}(0) 
\ge \dim \J_{m-2}(\mathcal{N}(\sl_2(\C))) + \dim \sl_2(\C) = 
2 (m-1) +3 = 2m +1, 
$$
whence the expected result 
since $\dim \J_m(\mathcal{N}(\sl_2(\C))) = 2(m+1) =2m+2$. 
\end{proof}

\begin{rem} \label{rk:norm} 
The equalities (\ref{eq2:norm}) for $m=1$ is also a consequence of 
Theorem~\ref{t:Mus},(3). 
\end{rem}

We now give an example illustrating that the converse of Proposition \ref{p:geo} 
is also not true for reducedness.  

\begin{ex} \label{ex:red}
The scheme $\J_1(\mathcal{N}(\sl_2(\C)))$ is irreducible and reduced. 
We readily obtain from the description of $\J_1(\mathcal{N}(\sl_2(\C)))$ 
given in Example \ref{ex:aff} that $\J_1(\J_1(\mathcal{N}(\sl_2(\C))))$ is defined by 
the ideal $\mathcal{J}$ of 
$$
\C[x_0,y_0,z_0,x_1,y_1,z_1,x'_0,y'_0,z'_0,x'_1,y'_1,z'_1]
$$ 
generated by the polynomials 
$$
x_0^2 + y_0z_0, \; 2x_0 x_1 +y_0 z_1+ y_1 z_0,
$$ 
$$
2 x_0 x'_0+ y_0 z'_0 + z_0 y'_0, \; 
2 x_0 x'_1 +2 x_1 x'_0 + y_0 z'_1 + y_1 z'_0 + z_1 y'_0 + z_0 y'_1.
$$
A computation made with the program \texttt{Macaulay2} shows that 
$\mathcal{J}$ is not radical, and that the radical of $\mathcal{J}$ is the 
intersection of two prime ideals. So, $\J_1(\J_1(\mathcal{N}(\sl_2(\C))))$ 
is neither reduced nor irreducible. 
\end{ex}

Example \ref{ex:red} gives another evidence that $\J_1(\mathcal{N}(\sl_2(\C)))$ 
does not have rational singularities (cf.~Proposition \ref{p:norm}). Indeed, if it had so, then 
by Theorem \ref{t:Mus}, $\J_1(\J_1(\mathcal{N}(\sl_2(\C))))$ would be irreducible 
(and reduced) because $\J_1(\mathcal{N}(\sl_2(\C)))$ is a complete intersection. 

We now turn to other interesting phenomena. 

\begin{ex} \label{ex2:2p}
As it has been observed in Example \ref{ex:2p}, for the nilpotent orbit $\O_{(2^p)}$ 
of $\sl_{2p}(\C)$, with $p\ge 2$, $\J_1(\bar{\O_{(2^p)}})$ is reducible and 
$$
\dim \pi_{\bar{\O_{(2^p)}},1}^{-1}\left((\bar{\O_{(2^p)}})_\sing\right)< 
\dim \pi_{\bar{\O_{(2^p)}},1}^{-1}(\O_{(2^p)}).
$$  
This shows that 
Lemma~\ref{l:irr},(3), does not hold in general if $X$ is not a complete intersection. 
\end{ex} 

\begin{ex} \label{ex:33}
As it has been observed in Remark \ref{rk:rect},(2), for the nilpotent orbit $\O_{(3^2)}$ 
of $\sl_6(\C)$, $\J_1(\bar{\O_{(3^2)}})$ is irreducible and  
$$\left(\J_1(\bar{\O_{(3^2)}})\right)_\reg \not= \pi_{\bar{\O_{(3^2)}},1}^{-1}(\O_{(3^2)}).$$ 
This shows that Theorem \ref{t:Mus},(3), is not true for non locally complete 
intersection varieties. 
\end{ex}

\begin{ex} \label{ex:22}
For the nilpotent orbit $\O_{(2^2)}$ of $\sp_4(\C)$, 
we have observed (cf.~Remark~\ref{rk2:sub}) that 
$$\left(\J_1(\bar{\O_{(2^2)}})\right)_\reg = \pi_{\bar{\O_{(2^2)}},1}^{-1}(\O_{(2^2)}).$$
This shows that the equality of Theorem \ref{t:Mus},(3) may hold 
even if $X$ is not locally a complete intersection. 
\end{ex}

%
%
%
%
%

\subsection*{Questions and remarks} 

Although we have determined the reducibliity of the closure of many nilpotent orbits,
we would like to complete the cases which our methods do not apply. Here are 
some open questions. 

\begin{question}
We have seen that jet schemes of nilpotent orbits in $\sl_n(\C )$ corresponding to rectangular 
partitions can be irreducible or reducible. Is there an explicit characterization ? 
\end{question}

\begin{question}
In all our examples of nilpotent orbits $\O$ with $\J_1(\bar{\O})$ reducible, the orbit $\O$ verifies \RC$_1$ or \RC$_2(1)$. 
Are these conditions necessary or are there examples of $\O$ for which $\J_1(\bar{\O})$ is reducible and that verify neither 
\RC$_1$ nor \RC$_2(1)$?
\end{question}

We have used the reducibility of jet schemes to study the property of complete intersection for nilpotent orbit closures. 
It is very likely  that  other geometrical properties of nilpotent orbit closures can be 
studied using jet schemes.

\appendix

\section{Nilpotent orbits in classical simple Lie algebras}   \label{app:not}
We fix in this appendix some notations, and basic results, relative to nilpotent 
orbits in simple Lie algebras of classical type. Our main references are 
\cite{CMa,Ke}. 
The results concerning the induction of nilpotent orbits are mostly 
taken from \cite{Ke}.

Let $n\in \N^{*}$, and denote by $\P(n)$ the set of partitions of $n$. 
As a rule, unless otherwise specified, we write an element $\bs{\lambda}$ of $\P(n)$ 
as a decreasing sequence $\bs{\lambda}=(\lambda_1,\ldots,\lambda_r)$ 
omitting the zeroes. 
Thus, 
$$
\lambda_1 \ge \cdots \ge \lambda_r \ge 1\quad \text{ and }\quad  
\lambda_1 + \cdots + \lambda_r  = n.
$$ 

We shall denote the dual partition of a partition $\bs{\lambda}\in\P(n)$ by $^t\!{\bs{\lambda}}$.
The {\em concatenation} of two partitions $\bs{\lambda}$ and $\bs{\lambda}'$ will be the rearrangement 
of the parts in decreasing order, and shall be denoted by 
$\bs{\lambda} \smile \bs{\lambda}'$. 

Let us denote by $\geqslant$ the partial order on $\P(n)$ relative to dominance. More precisely, given
$\bs{\lambda} = ( \lambda_{1},\cdots , \lambda_{r}) , \bs{\mu} = (\mu_{1},\dots ,\mu_{s}) \in \P (n)$,  
we have $\bs{\lambda}\ \geqslant \bs{\mu}$ if 
$$
\sum_{i=1}^{k} \lambda_{i} \geqslant \sum_{i=1}^{k} \mu_{i}
$$
for $1\leqslant k\leqslant \min (r,s)$.

\smallskip

\subsection*{Case $\sl_n(\C)$}~

According to \cite[Thm.\,5.1.1]{CMa}, nilpotent orbits of $\sl_{n}(\C)$ are 
parametrized by $\P(n)$. For $\bs{\lambda}\in \P(n)$, we shall denote by 
$\O_{\bs{\lambda}}$ the corresponding nilpotent orbit of $\sl_n(\C)$, and if we write
$^t\!{\bs{\lambda}} = ( d_{1}, \dots ,d_{s})$, then
$$
\dim \O_{\bs{\lambda}} = n^{2} - \sum_{i=1}^{s} d_{i}^{2}.
$$ 
Also, if $\bs{\lambda}, \bs{\mu} \in \P (n)$, then $\O_{\bs{\mu}} \subset \bar{\O_{\bs{\lambda}}}$ if and only if
$\bs{\mu} \leqslant \bs{\lambda}$.

The Levi subalgebras of $\sl_n(\C)$ are parametrized by compositions of $n$. 
Let $\mathbf{m} = (m_{1},\dots ,m_{r})$ be a composition of $n$, and 
let $\bs{\lambda}=(\bs{\lambda}^{(1)}, \dots ,\bs{\lambda}^{(r)}) \in \P(m_{1}) 
\times \cdots \times \P(m_{r})$. 
It corresponds to a nilpotent orbit in the Levi subalgebra associated to the composition 
$\mathbf{m}$.  Set 
$$
\bs{\mu} :=\, ^t\! \bs{\lambda} ^{(1)} \smile \cdots \smile
\,^t\! \bs{\lambda} ^{(r)} \quad \text{ and } \quad \bs{\nu} = \, 
^t \!{\bs{\mu}}.
$$
Then the partition associated to the induced nilpotent orbit from 
$\O_{(\bs{\lambda} ^{(1)}, \dots , \bs{\lambda} ^{(r)})}$ is $\bs{\nu}$. 
Note that we have $\nu_{i} = \lambda_{i}^{(1)} + \cdots + \lambda_{i}^{(k)}$ 
which is much simpler to compute in practice.
We shall denote $\bs{\nu}$ by 
${\rm Ind}_{\mathbf{m}}^{n} (\bs{\lambda} ^{(1)}, \dots ,\bs{\lambda} ^{(r)})$ and we shall say that 
$\bs{\nu}$ is {\em induced from} $(\bs{\lambda} ^{(1)}, \dots ,\bs{\lambda} ^{(r)})$.

\smallskip

\subsection*{Case $\so_n(\C)$}~

For $n\in\N^*$, set 
$$
\P_{1}(n):=\{\bs{\lambda} \in \P(n)\; ; \; \text{number of parts of each even 
number is even}\}.
$$ 
According to \cite[Thm.\,5.1.2 and 5.1.4]{CMa}, nilpotent orbits of $\so_{n}(\C)$ 
are parametrized by $\P(n)$, with the exception that each {\em very even} 
partition $\bs{\lambda} \in\P_{1}(n)$ (i.e., $\bs{\lambda}$ has only even parts) 
corresponds to two nilpotent orbits.
For $\bs{\lambda}\in \P_1(n)$, not very even, we shall denote by 
$\O_{\bs{\lambda}}$ the corresponding nilpotent orbit of $\so_n(\C)$. 
For very even $\bs{\lambda}\in \P_1(n)$, we shall denote by $\O_{\bs{\lambda}}^{I}$ 
and $\O_{\bs{\lambda}}^{I\!I}$ the two corresponding nilpotent orbits of $\so_n(\C)$.  
In fact, their union form a single ${\rm O}_{n}(\C)$-orbit. 

Let $\bs{\lambda} =(\lambda_{1},\dots ,\lambda_{r}) \in \P_1(n)$ and 
$^t\!{\bs{\lambda}} = (d_{1},\dots ,d_{s})$, then 
$$
\dim \O_{\bs{\lambda}}^{\bullet} = \frac{n(n-1)}{2} - \frac{1}{2}\left( \sum_{i=1}^{s} d_{i}^{2}
-  \sharp\{ i ; \lambda_{i} \hbox{ odd} \}  \right) 
$$
where $\O_{\bs{\lambda}}^{\bullet}$ is either $\O_{\bs{\lambda}}$, $\O_{\bs{\lambda}}^{I}$ or $\O_{\bs{\lambda}}^{II}$ according 
to whether $\bs{\lambda}$ is very even or not. Using the same notations, if $\bs{\lambda} , \bs{\mu} \in \P_{1} (n)$, then 
$\bar{\O_{\bs{\mu}}^{\bullet}} \subsetneq  \bar{\O_{\bs{\lambda}}^{\bullet}}$ if and only if
$\bs{\mu} < \bs{\lambda}$. 

Given $\bs{\lambda} \in\P(n)$, there exists a unique $\bs{\lambda} ^{+}
\in \P_{1}(n)$ such that $\bs{\lambda} ^{+} \le \bs{\lambda} $, and if 
$\bs{\mu}\in\P_{1}(n)$ verifies $\bs{\mu} \le\bs{\lambda} $, then $\bs{\mu}\le \bs{\lambda} ^{+}$.
More precisely, let $\bs{\lambda}  = (\lambda_{1},\dots ,\lambda_{n})$ (adding zeroes if necessary). 
If $\bs{\lambda} \in \P_{1}(n)$, then $\bs{\lambda} ^{+}  =\bs{\lambda} $. 
Otherwise if $\bs{\lambda} \not\in \P_{1}(n)$, set
$$
\bs{\lambda} ' = (\lambda_{1},\dots ,\lambda_{r},\lambda_{r+1}-1,\lambda_{r+2},\dots ,
\lambda_{s-1},\lambda_{s}+1,\lambda_{s+1},\dots, \lambda_{n})
$$
where $r$ is maximum such that $(\lambda_{1},\dots ,\lambda_{r}) \in 
\P_{1}(\lambda_{1}+\cdots + \lambda_{r})$, and $s$ is the index of the first even part
in $(\lambda_{r+2},\dots ,\lambda_{n})$. Note that $r=0$ if such a maximum does not 
exist, while $s$ is always defined. If $\bs{\lambda} '$ is not in $\P_{1}(n)$, then we repeat 
the process until we obtain an element of $\P_{1}(n)$ which will be our $\bs{\lambda} ^{+}$.

The Levi subalgebras in $\so_n(\C)$ are parametrized by 
$$
\mathcal{L}(n):= \left\{ (p_{1},\dots ,p_{k} ; r ) \, ; \, 2 \sum_{i=1}^{k} p_{i} + r = n \right\}.
$$
Let $(p_{1},\dots ,p_{k} ; r ) \in \mathcal{L}(n)$,
$(\bs{\lambda} ^{(1)}, \dots ,\bs{\lambda} ^{(k)} ) \in \P (p_{1})\times \cdots \times \P(p_{k})$ 
and $\bs{\mu} \in \P_{1}(r)$, and set
$$
\bs{\nu} := {\rm Ind}_{( p_{1},\dots ,p_{k} , r ,p_{k},\dots ,p_{1})}^{n} 
(\bs{\lambda}^{(1)}, \dots ,\bs{\lambda} ^{(k)} , \bs{\mu} , \bs{\lambda}^{(k)}, \dots,
\bs{\lambda} ^{(1)})
$$ 
in the notations of the $\sl_n(\C)$ case. 
Thus $\bs{\nu}$ is the partition associated to the nilpotent orbit in $\sl_{n}(\C)$ 
induced from the nilpotent orbit  in the Levi subalgebra of $\sl_{n}(\C)$ associated to 
the composition $( p_{1},\dots ,p_{k} , r ,p_{k},\dots ,p_{1})$ and 
the multi-partition $(\bs{\lambda}^{(1)}, \dots ,\bs{\lambda} ^{(k)} ,
 \bs{\mu}, \bs{\lambda}^{(k)},\dots ,\bs{\lambda} ^{(1)})$. The partition associated to the 
 nilpotent orbit induced from  
$(\bs{\lambda} ^{(1)}, \dots ,\bs{\lambda} ^{(k)}  ; \bs{\mu})$ is $\bs{\nu}^{+}$. 
We shall denote $\bs{\nu}^{+}$ by ${\rm Ind}_{(p_{1},\dots,p_{k} ; r)}^{n,+} 
(\bs{\lambda} ^{(1)}, \dots ,\bs{\lambda} ^{(k)}  ; \bs{\mu})$.
The partition $\bs{\lambda} \in \P_{1}(n)$ corresponds to a rigid orbit if and only if 
\begin{itemize}
\item[(i)\;] $\lambda_{i}-\lambda_{i+1} \le 1$ for all $i$, so the last part of $\bs{\lambda}$ is $1$.
\item[(ii)] No odd number occurs exactly twice in $\bs{\lambda} $.
\end{itemize}

Note that in the case of a very even partition $\bs{\lambda}$, $\bs{\nu}^{+}$ is also very even partition,
and we obtain both nilpotent orbits corresponding to $\bs{\nu}^{+}$ via induction of the nilpotent
orbits corresponding to $\bs{\lambda}$, cf.~\cite[Thm.\,7.3.3,(iii)]{CMa}. 

\smallskip

\subsection*{Case $\sp_{2n}(\C)$.} ~ 

For $n\in\N^*$, set 
$$
\P_{-1}(2n):=\{\bs{\lambda} \in \P(2n)\; ; \; \text{number of parts of each odd  
number is even}\}.
$$ 
According to  \cite[Thm.\,5.1.3]{CMa}, nilpotent orbits of $\sp_{2n}(\C)$ 
are parametrized by $\P_{-1}(2n)$. 
For $\bs{\lambda}=(\lambda_{1},\dots ,\lambda_{r})\in \P_{-1}(2n)$, we shall denote by 
$\O_{\bs{\lambda}}$ the corresponding nilpotent orbit of $\sp_{2n}(\C)$,
and if we write $^t\!{\bs{\lambda}} = (d_{1},\dots ,d_{s})$, then 
$$
\dim \O_{\bs{\lambda}} = n(2n+1) - \frac{1}{2}\left( \sum_{i=1}^{s} d_{i}^{2}
+  \sharp\{ i ; \lambda_{i} \hbox{ odd} \}  \right).
$$
As in the case of $\sl_{n}(\C )$, if $\bs{\lambda}, \bs{\mu} \in \P_{-1} (2n)$, then 
$\O_{\bs{\mu}} \subset \bar{\O_{\bs{\lambda}}}$ if and only if
$\bs{\mu} \leqslant \bs{\lambda}$. 

Given $\bs{\lambda}  \in \P(2n)$, there exists a unique $\bs{\lambda} ^{-} \in \P_{-1}(2n)$ 
such that $\bs{\lambda} ^{-} \le \bs{\lambda} $, and if $\bs{\mu}\in\P_{-1}(2n)$ verifies 
$\bs{\mu} \leq \bs{\lambda} $, then $\bs{\mu}\leq \bs{\lambda} ^{-}$. The construction of 
$\bs{\lambda} ^{-}$ is the same as in the orthogonal case except that $s$ is the index 
of the first odd part in $(\lambda_{r+2},\dots ,\lambda_{2n})$.

As in the orthogonal case, Levi subalgebras are parametrized by $\mathcal{L}(2n)$. 
Let us conserve the same notations as in the orthogonal case. 
The partition associated to the nilpotent orbit induced from $(\bs{\lambda} ^{(1)}, \dots ,\bs{\lambda} ^{(k)} ; \bs{\mu})$ is $\bs{\nu}^{-}$. 
We shall denote $\bs{\nu}^{-}$ by ${\rm Ind}_{(p_{1},\dots,p_{k} ; r)}^{2n,-} 
(\bs{\lambda} ^{(1)}, \dots ,\bs{\lambda} ^{(k)}  ; \bs{\mu})$. 
The partition $\bs{\lambda} \in\P_{-1}(2n)$ corresponds to a rigid orbit if and only if 
\begin{itemize}
\item[(i)\;] $\lambda_{i}-\lambda_{i+1} \leqslant 1$ for all $i$, so the last part of 
$\bs{\lambda} $ is $1$.
\item[(ii)] No even number occurs exactly twice in $\bs{\lambda} $.
\end{itemize}

\section{Statistics in types $\bf{B}$, $\bf{C}$ and $\bf{D}$}\label{app:stat} 
As mentioned in Remark \ref{rk:BCD}, many nilpotent orbits in $\so_{n}(\C)$ and $\sp_{2n}(\C)$
can be obtained by induction from little nilpotent orbits. In particular, these induced orbits 
verify \RC$_{2}(m)$ for all $m\in\N^{*}$. Computations using \texttt{GAP4} gave us the following 
numerical data supporting our claim. 

For  $\varepsilon\in\{ -1, 1\}$ and $n\in\mathbb{N}^{*}$,
we denote by $\P_{\varepsilon}^{\ell}(n)$ the set of partitions in $\P_{\varepsilon}(n)$
induced from little ones.

\smallskip

\subsection*{Case $\so_{n}(\C )$}~

{\scriptsize
$$
\begin{array}{||c|c|c||c|c|c||c|c|c||c|c|c||c|c|c||}
n  & \sharp \P_{1}^{\ell}(n) & \sharp\P_{1}(n) &
n  & \sharp \P_{1}^{\ell}(n) & \sharp\P_{1}(n) &
n  & \sharp \P_{1}^{\ell}(n) & \sharp\P_{1}(n) &
n  & \sharp \P_{1}^{\ell}(n) & \sharp\P_{1}(n) &
n  & \sharp \P_{1}^{\ell}(n) & \sharp\P_{1}(n) 
 \\ [2pt] \hline
2 & 0 & 1 & 12 & 20 & 28 & 22 & 195 & 236 & 32 & 1223 & 1431 & 42 & 6064 & 6868 \\ \hline 
3 & 0 & 2 & 13 & 27 & 35 & 23 & 250 & 287 & 33 & 1474 & 1687 & 43 & 7086 & 7967 \\ \hline 
4 & 1 & 3 & 14 & 32 & 43 & 24 & 291 & 350 & 34 & 1710 & 1981 & 44 & 8182 & 9233 \\ \hline 
5 & 1 & 4 & 15 & 45 & 55 & 25 & 367 & 420 & 35 & 2039 & 2331 & 45 & 9536 & 10670 \\ \hline 
6 & 2 & 5 & 16 & 52 & 70 & 26 & 423 & 501 & 36 & 2370 & 2741 & 46 & 10986 & 12306 \\ \hline 
7 & 4 & 7 & 17 & 73 & 86 & 27 & 527 & 602 & 37 & 2821 & 3206 & 47 & 12748 & 14193 \\ \hline 
8 & 6 & 10 & 18 & 83 & 105 & 28 & 609 & 722 & 38 & 3265 & 3740 & 48 & 14667 & 16357 \\ \hline 
9 & 9 & 13 & 19 & 111 & 130 & 29 & 751 & 858 & 39 & 3852 & 4368 & 49 & 16974 & 18803 \\ \hline 
10 & 10 & 16 & 20 & 130 & 161 & 30 & 869 & 1016 & 40 & 4460 & 5096 & 50 & 19485 & 21581 \\ \hline 
11 & 16 & 21 & 21 & 170 & 196 & 31 & 1055 & 1206 & 41 & 5242 & 5922 & 51 & 22464 & 24766 
\end{array}
$$
}
\subsection*{Case $\sp_{2n}(\C )$} ~

{\scriptsize
$$
\begin{array}{||c|c|c||c|c|c||c|c|c||c|c|c||}
n  & \sharp \P_{-1}^{\ell}(2n) & \sharp\P_{-1}(2n) &
n  & \sharp \P_{-1}^{\ell}(2n) & \sharp\P_{-1}(2n) &
n  & \sharp \P_{-1}^{\ell}(2n) & \sharp\P_{-1}(2n) &
n  & \sharp \P_{-1}^{\ell}(2n) & \sharp\P_{-1}(2n) 
 \\ [2pt] \hline
1 & 0 & 2 & 7 & 45 & 64 & 13 & 594 & 728 & 19 & 4652 & 5400 \\ \hline 
2 & 1 & 4 & 8 & 77 & 100 & 14 & 857 & 1040 & 20 & 6374 & 7336 \\ \hline 
3 & 3 & 8 & 9 & 119 & 154 & 15 & 1223 & 1472 & 21 & 8677 & 9904 \\ \hline 
4 & 9 & 14 & 10 & 182 & 232 & 16 & 1726 & 2062 & 22 & 11728 & 13288 \\ \hline 
5 & 15 & 24 & 11 & 273 & 344 & 17 & 2421 & 2864 & 23 & 15755 & 17728 \\ \hline 
6 & 28 & 40 & 12 & 409 & 504 & 18 & 3378 & 3948 & 24 & 21061 & 23528
\end{array}
$$
}

\section{Tables for exceptional types} \label{app:exc}
We list below nilpotent orbits in a simple Lie algebra of exceptional type precising when possible
whether they are \RC$_{1}$ or \RC$_{2}(m)$. Condition \RC$_{1}$ is checked using Propositions \ref{p:eta}, \ref{p:litt},
\ref{p:res} and Remark \ref{rk:eta}. When condition \RC$_{1}$ is obtained via Proposition \ref{p:res}, we give
an example of the bigger simple Lie algebra and the little nilpotent orbit 
satisfying condition (iii) of Proposition~\ref{p:res} from which it is obtained. 

As for determining whether condition \RC$_{2}(m)$ is verified, our main method is to
list the orbits induced by nilpotent orbits that have a little factor (Theorem \ref{t2:ind}). 
Thus they are \RC$_{2}(m)$ for all $m\in\N^{*}$.
Since induction is transitive, we can proceed by induction on 
the rank of the Lie algebra, where at each step, we only need to consider induction from orbits in maximal Levi subalgebras 
which are themselves induced from nilpotent orbits with a little factor. 
For an orbit verifying condition \RC$_{2}(m)$, we give an example of 
a maximal Levi subalgebra $\l$ and an orbit in $\l$ induced from a nilpotent orbit with a little factor.  

In both cases, if the orbit is little, then we just label it little. The subscript of an orbit 
indicates either its characteristics or the associated partition or its Bala-Carter label. 
If a superscript of an orbit is present, it indicates the corresponding maximal Levi subalgebra.

We have omitted the zero orbit and the regular orbit because they are neither  \RC$_{1}$ nor \RC$_{2}(m)$.

All the computations are done using the package \texttt{sla} of \texttt{GAP4}.

\subsection*{Type $\bf{G_{2}}$}~

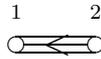
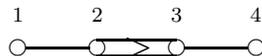
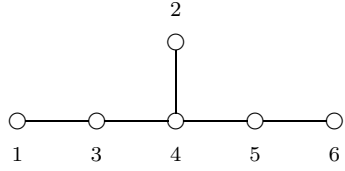
$$
\begin{Dynkin}
\Dbloc{\Dcirc\Ddoubleeast\Deast\Dtext{t}{1}}\Dleftarrow
\Dbloc{\Dcirc\Ddoublewest\Dwest\Dtext{t}{2}}
\end{Dynkin}
$$

$$
\begin{array}{c|c|c|c|c|c}
\multicolumn{2}{c|}{\O} & \dim \O & RC_{1} & RC_{2} & \hbox{rigid}  \\ \hline
\tvi A_{1} & [0,1] & 6 & \oui \leftarrow \hbox{little} & \oui \leftarrow \hbox{little} & \oui  \\ \hline
\tvi \tilde{A}_{1} & [1,0] & 8 & \non  & ?  & \oui\\ \hline
\tvi G_{2}(a_{1}) & [2,0] & 10 & \non & ? & \non
\end{array}
$$

\noindent\hrulefill
\bigskip

\subsection*{Type $\bf{F_{4}}$}~
$$
\begin{Dynkin}
\Dbloc{\Dcirc\Deast\Dtext{t}{1}}
\Dbloc{\Dcirc\Dwest\Ddoubleeast\Dtext{t}{2}}\Drightarrow
\Dbloc{\Dcirc\Ddoublewest\Deast\Dtext{t}{3}}
\Dbloc{\Dcirc\Dwest\Dtext{t}{4}}
\end{Dynkin}
$$

$$
\begin{array}{c|c|c|c|c|c}
\multicolumn{2}{c|}{\O}&  \dim \O & RC_{1} & RC_{2} & \hbox{rigid}  \\ \hline
\tvi A_{1} & [ 1, 0, 0, 0 ] & 16 & \oui \leftarrow \hbox{little} & \oui \leftarrow \hbox{little} & \oui   \\ \hline
\tvi \tilde{A_{1}} & [ 0, 0, 0, 1 ] & 22 & \oui \leftarrow \hbox{little}  & \oui \leftarrow \hbox{little} &  \oui     \\ \hline
\tvi A_{1} + \tilde{A}_{1} & [ 0, 1, 0, 0 ] & 28 & \non & ? & \oui  \\ \hline
\tvi A_{2} & [ 2, 0, 0, 0 ] & 30 &\non  &  ?  & \non \\ \hline 
\tvi \tilde{A}_{2} & [ 0, 0, 0, 2 ] & 30 & \non  &  ? & \non \\ \hline
\tvi A_{2} + \tilde{A}_{1} & [ 0, 0, 1, 0 ] & 34 & \non   & ? & \oui \\ \hline
\tvi B_{2} & [ 2, 0, 0, 1 ] & 36 & \non   & \oui \leftarrow  \O_{\rm min}^{\{2,3,4\}} & \non \\ \hline
\tvi \tilde{A}_{2} + A_{1} & [ 0, 1, 0, 1 ] & 36 & \non & ?  & \oui \\ \hline
\tvi C_{3}(a_{1}) & [ 1, 0, 1, 0 ] & 38 & \non   &  \oui \leftarrow  \O_{\rm min}^{\{1,2,3\}} & \non \\ \hline
\tvi F_{4}(a_{3}) & [ 0, 2, 0, 0 ] & 40 & \non   & \oui \leftarrow  \O_{[0,1,0]}^{\{2,3,4\}} & \non  \\ \hline
\tvi B_{3} & [ 2, 2, 0, 0 ] & 42 & \non   & ?  & \non  \\ \hline
\tvi C_{3} & [ 1, 0, 1, 2 ] & 42 & \non  &  ? & \non  \\ \hline
\tvi F_{4}(a_{2}) & [ 0, 2, 0, 2 ] & 44 & \non   & \oui \leftarrow \O_{(2,1),(1^{2})}^{\{1,2,4\}} & \non  \\ \hline
\tvi F_{4}(a_{1}) & [ 2, 2, 0, 2 ] & 46 &\non  & \oui \leftarrow \O_{(2,1),(2)}^{\{1,2,4\}} & \non
\end{array}
$$

\noindent\hrulefill
\newpage

\subsection*{Type $\bf{E_{6}}$}~

$$
\begin{Dynkin}
\Dspace\Dspace\Dbloc{\Dcirc\Dsouth\Dtext{t}{2}}
\Dskip
\Dbloc{\Dcirc\Deast\Dtext{b}{1}}
\Dbloc{\Dcirc\Dwest\Deast\Dtext{b}{3}}
\Dbloc{\Dcirc\Dwest\Deast\Dnorth\Dtext{b}{4}}
\Dbloc{\Dcirc\Dwest\Deast\Dtext{b}{5}}
\Dbloc{\Dcirc\Dwest\Dtext{b}{6}}
\end{Dynkin}
$$

\medskip

$$
\begin{array}{c|c|c|c|c|c}
\multicolumn{2}{c|}{\O}&  \dim \O & RC_{1}  & RC_{2} & \hbox{rigid} \\ \hline
\tvi A_{1} & [ 0, 1, 0, 0, 0, 0 ] & 22 & \oui \leftarrow \hbox{little} & \oui \leftarrow \hbox{little} & \oui \\ \hline  
\tvi 2A_{1} & [ 1, 0, 0, 0, 0, 1 ] &  32 & \oui \leftarrow \hbox{little}  & \oui \leftarrow \hbox{little} & \non  \\ \hline  
\tvi 3A_{1} & [ 0, 0, 0, 1, 0, 0 ] & 40 & \oui \leftarrow  \mathrm{Res}_{E_{6}}^{E_{7}} \O_{(3A_{1})'}  & ? & \oui \\ \hline  
\tvi A_{2} &  [ 0, 2, 0, 0, 0, 0 ]  & 42 & \oui \leftarrow \mathrm{Res}_{E_{6}}^{E_{7}} \O_{A_{2}} &  ? & \non  \\ \hline 
\tvi A_{2}+A_{1} & [ 1, 1, 0, 0, 0, 1 ] & 46 & \non  & \oui \leftarrow \O_{min}^{\{1,2,3,4,5\}} & \non \\ \hline  
\tvi 2A_{2} &  [ 2, 0, 0, 0, 0, 2 ] & 48 & \non & \oui \leftarrow \O_{(3,1^{7})}^{\{1,2,3,4,5\}} & \non \\ \hline 
\tvi A_{2} + 2A_{1} &  [ 0, 0, 1, 0, 1, 0 ] & 50 & \non & ? & \non \\ \hline  
\tvi A_{3} &   [ 1, 2, 0, 0, 0, 1 ] & 52 & \non & \oui \leftarrow \O_{(2,1^{4})}^{\{1,3,4,5,6\}} & \non \\ \hline  
\tvi 2A_{2} +A_{1} &  [ 1, 0, 0, 1, 0, 1 ] & 54&\non & ? & \oui \\ \hline  
\tvi A_{3} +A_{1} &  [ 0, 1, 1, 0, 1, 0 ] & 56 &\non & ? & \non \\ \hline
\tvi D_{4}(a_{1}) &  [ 0, 0, 0, 2, 0, 0 ] & 58 &\non  & \oui \leftarrow \O_{(2^{2},1^{2})}^{\{1,3,4,5,6\}} & \non \\ \hline  
\tvi A_{4} &  [ 2, 2, 0, 0, 0, 2 ] & 60 &\non & \oui \leftarrow \O_{(3,1^{3})}^{\{1,3,4,5,6\}} & \non \\ \hline  
\tvi D_{4} &  [ 0, 2, 0, 2, 0, 0 ] & 60 &\non & ? & \non  \\ \hline  
\tvi A_{4} + A_{1} &  [ 1, 1, 1, 0, 1, 1 ] & 62 &\non  &  \oui \leftarrow \O_{( (2,1^{2}) ,(1^{2}) )}^{\{1,2,3,4,6\}} & \non \\ \hline  
\tvi A_{5} &  [ 2, 1, 1, 0, 1, 2 ] & 64 &\non  & ? & \non \\ \hline  
\tvi D_{5}(a_{1}) & [ 1, 2, 1, 0, 1, 1 ] & 64 &\non  &  \oui \leftarrow \O_{(3,2,1)}^{\{1,3,4,5,6\}} & \non \\ \hline  
\tvi E_{6}(a_{3}) &  [ 2, 0, 0, 2, 0, 2 ] & 66 &\non  &  \oui \leftarrow \O_{(4,1^{2})}^{\{1,3,4,5,6\}} & \non \\ \hline  
\tvi D_{5} &  [ 2, 2, 0, 2, 0, 2 ] & 68 &\non  &  \oui \leftarrow \O_{(4,2)}^{\{1,3,4,5,6\}} & \non \\ \hline  
\tvi E_{6}(a_{1}) &  [ 2, 2, 2, 0, 2, 2  ] & 70 &\non  & \oui \leftarrow \O_{(5,1)}^{\{1,3,4,5,6\}} & \non
\end{array}
$$

\medskip

The notation $\mathrm{Res}_{E_{6}}^{E_{7}} \O$ means that the orbit is obtained by 
restriction from the little nilpotent orbit $\O$ in $E_{7}$ as explained in Table \ref{tab:res}.

\medskip
\noindent\hrulefill

\newpage

\subsection*{Type $\bf{E_{7}}$}~

$$
\begin{Dynkin}
\Dspace\Dspace\Dbloc{\Dcirc\Dsouth\Dtext{t}{2}}
\Dskip
\Dbloc{\Dcirc\Deast\Dtext{b}{1}}
\Dbloc{\Dcirc\Dwest\Deast\Dtext{b}{3}}
\Dbloc{\Dcirc\Dwest\Deast\Dnorth\Dtext{b}{4}}
\Dbloc{\Dcirc\Dwest\Deast\Dtext{b}{5}}
\Dbloc{\Dcirc\Dwest\Deast\Dtext{b}{6}}
\Dbloc{\Dcirc\Dwest\Dtext{b}{7}}
\end{Dynkin}
$$
\bigskip

$$
\begin{array}{c|c|c|c|c|c}
\multicolumn{2}{c|}{\O}&  \dim \O & RC_{1} & RC_{2} & \hbox{rigid} \\ \hline
\tvi A_{1} & [ 1, 0, 0, 0, 0, 0, 0 ] & 34 & \oui \leftarrow \hbox{little} &  \oui \leftarrow \hbox{little} & \oui \\ \hline
\tvi 2A_{1} & [ 0, 0, 0, 0, 0, 1, 0 ] & 52 &  \oui \leftarrow \hbox{little}  &  \oui  \leftarrow \hbox{little} & \oui \\ \hline
\tvi (3A_{1})'' & [ 0, 0, 0, 0, 0, 0, 2 ] & 54 & \oui \leftarrow \hbox{little}  & \oui \leftarrow \hbox{little} & \non \\ \hline
\tvi (3A_{1})' & [ 0, 0, 1, 0, 0, 0, 0 ] & 64 &  \oui \leftarrow \hbox{little}  &  \oui \leftarrow \hbox{little} & \oui  \\ \hline
\tvi A_{2} & [ 2, 0, 0, 0, 0, 0, 0 ] & 66 & \oui \leftarrow \hbox{little}  & \oui \leftarrow \hbox{little} & \non \\ \hline
\tvi 4A_{1} & [ 0, 1, 0, 0, 0, 0, 1 ] & 70 & \non & ? & \oui \\ \hline
\tvi A_{2}+A_{1} & [ 1, 0, 0, 0, 0, 1, 0 ] & 76 & \non & \oui \leftarrow \O^{\{1,2,3,4,5,6\}}_{[0,1,0,0,0,0]} & \non \\ \hline
\tvi A_{2}+2A_{1} & [ 0, 0, 0, 1, 0, 0, 0 ] & 82 & \non & ? & \oui \\ \hline
\tvi A_{3} & [ 2, 0, 0, 0, 0, 1, 0 ] & 84 & \non & \oui \leftarrow \O^{\{2,3,4,5,6,7\}}_{( 2^{2},1^{8})} & \non \\ \hline
\tvi 2A_{2} & [ 0, 0, 0, 0, 0, 2, 0 ] & 84 & \non & ? & \non \\ \hline
\tvi A_{2}+3A_{1} & [ 0, 2, 0, 0, 0, 0, 0 ] & 84 & \non & ? & \non \\ \hline
\tvi (A_{3}+A_{1})'' & [ 2, 0, 0, 0, 0, 0, 2 ] & 86 & \non & \oui \leftarrow \O^{\{2,3,4,5,6,7\}}_{( 3,1^{9})} & \non \\ \hline
\tvi 2A_{2}+A_{1} & [ 0, 0, 1, 0, 0, 1, 0 ] & 90 & \non & ? & \oui  \\ \hline
\tvi (A_{3}+A_{1})' & [ 1, 0, 0, 1, 0, 0, 0 ] & 92 & \non & ? & \oui \\ \hline
\tvi D_{4}(a_{1}) & [ 0, 0, 2, 0, 0, 0, 0 ] & 94 & \non &  \oui \leftarrow \O^{\{2,3,4,5,6,7\}}_{( 2^{4},1^{4})} & \non  \\ \hline
\tvi A_{3} + 2A_{1} & [ 1, 0, 0, 0, 1, 0, 1 ] & 94 & \non & ? & \non \\ \hline
\tvi D_{4} & [ 2, 0, 2, 0, 0, 0, 0 ] & 96 & \non &  \oui \leftarrow \O^{\{2,3,4,5,6,7\}}_{[0,0,0,0,2,0]} & \non \\ \hline
\tvi D_{4}(a_{1})+A_{1} & [ 0, 1, 1, 0, 0, 0, 1 ] & 96 & \non &  \oui \leftarrow \O^{\{2,3,4,5,6,7\}}_{[0,0,0,0,0,2]} & \non \\ \hline
\tvi A_{3}+A_{2} & [ 0, 0, 0, 1, 0, 1, 0 ] & 98 & \non & \oui \leftarrow \O^{\{2,3,4,5,6,7\}}_{( 3,2^{2}, 1^{5})} & \non \\ \hline
\tvi A_{4} & [ 2, 0, 0, 0, 0, 2, 0 ] & 100 & \non &  \oui \leftarrow \O^{\{2,3,4,5,6,7\}}_{( 3^{2},1^{6})} & \non \\ \hline
\tvi A_{3}+A_{2}+A_{1} & [ 0, 0, 0, 0, 2, 0, 0 ] & 100 & \non & ? & \non  \\ \hline
\tvi (A_{5})'' & [ 2, 0, 0, 0, 0, 2, 2 ] & 102 & \non &  \oui \leftarrow \O^{\{2,3,4,5,6,7\}}_{( 5,1^{7})} & \non
\end{array}
$$

\medskip

Note that the characteristics $[0,0,0,0,2,0]$ and $[0,0,0,0,0,2]$ of nilpotent orbits in $\bf{D_{6}}$ correspond to the very even partition $(2^{6})$.

\subsection*{Type $\bf{E_{7}}$ (cont'd)} ~

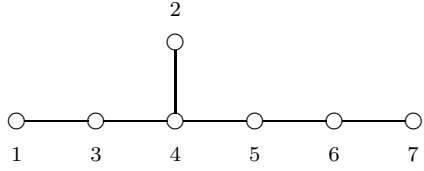
$$
\begin{Dynkin}
\Dspace\Dspace\Dbloc{\Dcirc\Dsouth\Dtext{t}{2}}
\Dskip
\Dbloc{\Dcirc\Deast\Dtext{b}{1}}
\Dbloc{\Dcirc\Dwest\Deast\Dtext{b}{3}}
\Dbloc{\Dcirc\Dwest\Deast\Dnorth\Dtext{b}{4}}
\Dbloc{\Dcirc\Dwest\Deast\Dtext{b}{5}}
\Dbloc{\Dcirc\Dwest\Deast\Dtext{b}{6}}
\Dbloc{\Dcirc\Dwest\Dtext{b}{7}}
\end{Dynkin}
$$
\bigskip

$$
\begin{array}{c|c|c|c|c|c}
\multicolumn{2}{c|}{\O}&  \dim \O & RC_{1} & RC_{2} & \hbox{rigid} \\ \hline
\tvi D_{4}+A_{1} & [ 2, 1, 1, 0, 0, 0, 1 ] & 102 & \non & ? & \non \\ \hline
\tvi A_{4}+A_{1} & [ 1, 0, 0, 1, 0, 1, 0 ] & 104 & \non &  \oui \leftarrow \O^{\{1,3,4,5,6,7\}}_{( 2^{2},1^{3})} & \non \\ \hline
\tvi D_{5}(a_{1}) & [ 2, 0, 0, 1, 0, 1, 0 ] & 106 & \non &  \oui \leftarrow \O^{\{2,3,4,5,6,7\}}_{( 3^{2},2^{2},1^{2})} & \non \\ \hline
\tvi A_{4}+A_{2} & [ 0, 0, 0, 2, 0, 0, 0 ] & 106 & \non & ? & \non \\ \hline
\tvi (A_{5})' & [ 1, 0, 0, 1, 0, 2, 0 ] & 108 & \non & ? & \non \\ \hline
\tvi A_{5}+A_{1} & [ 1, 0, 0, 1, 0, 1, 2 ] & 108 & \non & ? & \non \\ \hline
\tvi D_{5}(a_{1}) +A_{1} & [ 2, 0, 0, 0, 2, 0, 0 ] & 108 & \non &  \oui \leftarrow \O^{\{1,2,3,4,6,7\}}_{(2,1^{3}),(1^{3})} & \non \\ \hline
\tvi D_{5}(a_{2}) & [ 0, 1, 1, 0, 1, 0, 2 ] & 110 & \non & ? & \non \\ \hline
\tvi E_{6}(a_{3}) & [ 0, 0, 2, 0, 0, 2, 0 ] & 110 & \non & \oui \leftarrow \O^{\{1,2,3,5,6,7\}}_{(2,1),(1^{2}), (1^{4})} & \non \\ \hline
\tvi D_{5} & [ 2, 0, 2, 0, 0, 2, 0 ] & 112 & \non & \oui \leftarrow \O^{\{2,3,4,5,6,7\}}_{[0,2,0,0,2,0]} & \non \\ \hline
\tvi E_{7}(a_{5}) & [ 0, 0, 0, 2, 0, 0, 2 ] & 112 & \non &  \oui \leftarrow \O^{\{2,3,4,5,6,7\}}_{[0,2,0,0,0,2]} & \non \\ \hline
\tvi A_{6} & [ 0, 0, 0, 2, 0, 2, 0 ] & 114 & \non & ?  & \non \\ \hline
\tvi D_{5}+A_{1} & [ 2, 1, 1, 0, 1, 1, 0 ] & 114 & \non & \oui \leftarrow \O^{\{1,2,3,4,6,7\}}_{(3,1^{2}),(1^{3})} & \non \\ \hline
\tvi D_{6}(a_{1}) & [ 2, 1, 1, 0, 1, 0, 2 ] & 114 & \non & \oui \leftarrow \O^{\{1,2,4,5,6,7\}}_{(2),( 3,1^{3})}  & \non \\ \hline
\tvi E_{7}(a_{4}) & [ 2, 0, 0, 2, 0, 0, 2 ] & 116 & \non & \oui \leftarrow \O^{\{1,3,4,5,6,7\}}_{( 3^{2},1)} & \non \\ \hline
\tvi D_{6} & [ 2, 1, 1, 0, 1, 2, 2 ] & 118 & \non & ? & \non \\ \hline
\tvi E_{6}(a_{1}) & [ 2, 0, 0, 2, 0, 2, 0 ] & 118 & \non & \oui \leftarrow \O^{\{2,3,4,5,6,7\}}_{( 5^{2},1^{2})} & \non  \\ \hline
\tvi E_{6} & [ 2, 0, 2, 2, 0, 2, 0 ] & 120 & \non &  \oui \leftarrow \O^{\{2,3,4,5,6,7\}}_{[0,2,0,2,2,0]}  & \non \\ \hline
\tvi E_{7}(a_{3}) & [ 2, 0, 0, 2, 0, 2, 2 ] & 120 & \non &   \oui \leftarrow \O^{\{2,3,4,5,6,7\}}_{[0,2,0,2,0,2]} & \non \\ \hline
\tvi E_{7}(a_{2}) & [ 2, 2, 2, 0, 2, 0, 2 ] & 122 & \non &  \oui \leftarrow \O^{\{1,3,4,5,6,7\}}_{( 5,2)} & \non \\ \hline
\tvi E_{7}(a_{1}) & [ 2, 2, 2, 0, 2, 2, 2 ] & 124 & \non &  \oui \leftarrow \O^{\{1,3,4,5,6,7\}}_{( 6,1)} & \non
\end{array}
$$
\medskip

Note that the characteristics $[0,2,0,0,2,0]$ and $[0,2,0,0,0,2]$ of nilpotent orbits in $\bf{D_{6}}$ correspond to the 
very even partition $(4^{2},2^{2})$,
while $[0,2,0,2,2,0]$ and $[0,2,0,2,0,2]$ correspond to $(6^{2})$.

\medskip
\noindent\hrulefill
\newpage

\subsection*{Type $\bf{E_{8}}$}~

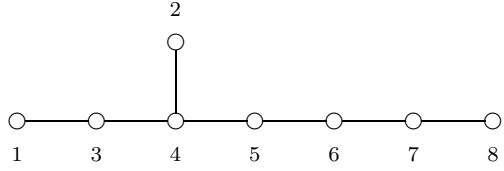
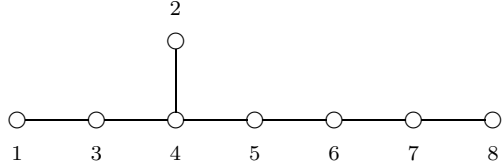
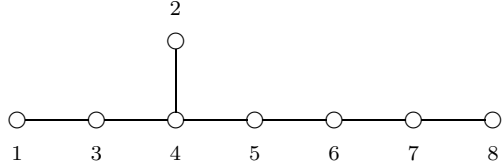
$$
\begin{Dynkin}
\Dspace\Dspace\Dbloc{\Dcirc\Dsouth\Dtext{t}{2}}
\Dskip
\Dbloc{\Dcirc\Deast\Dtext{b}{1}}
\Dbloc{\Dcirc\Dwest\Deast\Dtext{b}{3}}
\Dbloc{\Dcirc\Dwest\Deast\Dnorth\Dtext{b}{4}}
\Dbloc{\Dcirc\Dwest\Deast\Dtext{b}{5}}
\Dbloc{\Dcirc\Dwest\Deast\Dtext{b}{6}}
\Dbloc{\Dcirc\Dwest\Deast\Dtext{b}{7}}
\Dbloc{\Dcirc\Dwest\Dtext{b}{8}}
\end{Dynkin}
$$

$$
\begin{array}{c|c|c|c|c|c}
\multicolumn{2}{c|}{\O}&  \dim \O & RC_{1} & RC_{2} & \hbox{rigid} \\ \hline
\tvi  A_{1} &  [ 0,0,0,0,0,0,0,1]  &  58  &   \oui \leftarrow \hbox{little} & \oui \leftarrow \hbox{little} & \oui \\ \hline
\tvi 2A_{1}  &  [ 1,0,0,0,0,0,0,0]  &  92 &   \oui \leftarrow \hbox{little} & \oui \leftarrow \hbox{little} & \oui \\ \hline
\tvi 3A_{1} &  [ 0,0,0,0,0,0,1,0]  &  112 &   \oui \leftarrow \hbox{little} & \oui \leftarrow \hbox{little} & \oui \\ \hline
\tvi A_{2} &  [ 0,0,0,0,0,0,0,2]  &  114  &   \oui \leftarrow \hbox{little}  & \oui \leftarrow \hbox{little} & \non \\ \hline
\tvi 4A_{1} &  [ 0,1,0,0,0,0,0,0]  &  128  &   \non & ? & \oui \\ \hline
\tvi A_{2}+A_{1} &  [ 1,0,0,0,0,0,0,1]  &  136  &   \non & ? & \oui \\ \hline
\tvi A_{2}+2A_{1} &  [ 0,0,0,0,0,1,0,0]  &  146  &   \non & ? & \oui \\ \hline
\tvi A_{3} &  [ 1,0,0,0,0,0,0,2]  &  148  &   \non &  \oui \leftarrow \O_{[1,0,0,0,0,0,0]}^{\{1,2,3,4,5,6,7\} } & \non \\ \hline
\tvi A_{2}+3A_{1} &  [ 0,0,1,0,0,0,0,0]  &  154  &   \non & ? & \oui \\ \hline
\tvi 2A_{2} &  [ 2,0,0,0,0,0,0,0]  &  156  &   \non & ? & \non \\ \hline
\tvi 2A_{2}+A_{1} &  [ 1,0,0,0,0,0,1,0]  &  162 &   \non & ? & \oui \\ \hline
\tvi A_{3}+A_{1} &  [ 0,0,0,0,0,1,0,1]  &  164  &   \non & ? & \oui \\ \hline
\tvi D_{4}(a_{1}) &  [ 0,0,0,0,0,0,2,0]  &  166  &   \non &  \oui \leftarrow \O_{[0,0,0,0,0,1,0]}^{\{1,2,3,4,5,6,7\} }  & \non \\ \hline
\tvi D_{4} &  [ 0,0,0,0,0,0,2,2]  &  168  &   \non &  \oui \leftarrow \O_{[0,0,0,0,0,0,2]}^{\{ 1,2,3,4,5,6,7\} }  & \non \\ \hline
\tvi 2A_{2}+2A_{1} &  [ 0,0,0,0,1,0,0,0]  &  168  &   \non & ? & \oui \\ \hline
\tvi A_{3}+2A_{1} &  [ 0,0,1,0,0,0,0,1]  &  172 &   \non & ? & \oui \\ \hline
\tvi D_{4}(a_{1})+A_{1} &  [ 0,1,0,0,0,0,1,0]  &  176  &   \non & ? & \oui \\ \hline
\tvi A_{3}+A_{2} &  [ 1,0,0,0,0,1,0,0]  &  178  &   \non & \oui \leftarrow \O_{(2^{2},1^{10})}^{\{ 2,3,4,5,6,7,8\} } & \non \\ \hline
\tvi A_{4} &  [ 2,0,0,0,0,0,0,2]  &  180 &   \non & \oui \leftarrow \O_{(3,1^{11})}^{\{ 2,3,4,5,6,7,8\} } & \non \\ \hline
\tvi A_{3}+A_{2}+A_{1} &  [ 0,0,0,1,0,0,0,0]  &  182 &   \non & ? &  \oui \\ \hline
\tvi D_{4}+A_{1} &  [ 0,1,0,0,0,0,1,2]  &  184  &   \non & ? & \non \\ \hline
\tvi D_{4}(a_{1})+A_{2} &  [ 0,2,0,0,0,0,0,0]  &  184  &   \non & ? & \non \\ \hline
\tvi A_{4}+A_{1} &  [ 1,0,0,0,0,1,0,1]  &  188  &   \non &  \oui \leftarrow \O_{[0,1,0,0,0,0],(1^{2})}^{\{1,2,3,4,5,6,8\} } & \non
\end{array}
$$

\newpage

\subsection*{Type $\bf{E_{8}}$ (cont'd)}~

$$
\begin{Dynkin}
\Dspace\Dspace\Dbloc{\Dcirc\Dsouth\Dtext{t}{2}}
\Dskip
\Dbloc{\Dcirc\Deast\Dtext{b}{1}}
\Dbloc{\Dcirc\Dwest\Deast\Dtext{b}{3}}
\Dbloc{\Dcirc\Dwest\Deast\Dnorth\Dtext{b}{4}}
\Dbloc{\Dcirc\Dwest\Deast\Dtext{b}{5}}
\Dbloc{\Dcirc\Dwest\Deast\Dtext{b}{6}}
\Dbloc{\Dcirc\Dwest\Deast\Dtext{b}{7}}
\Dbloc{\Dcirc\Dwest\Dtext{b}{8}}
\end{Dynkin}
$$

$$
\begin{array}{c|c|c|c|c|c}
\multicolumn{2}{c|}{\O}&  \dim \O & RC_{1} & RC_{2} & \hbox{rigid} \\ \hline
\tvi 2A_{3} &  [ 1,0,0,0,1,0,0,0]  &  188  &   \non & ? & \oui  \\ \hline
\tvi D_{5}(a_{1}) &  [ 1,0,0,0,0,1,0,2]  &  190 &   \non & \oui \leftarrow \O_{(2^{4},1^{6})}^{\{2,3,4,5,6,7,8\} } & \non \\ \hline
\tvi A_{4}+2A_{1} &  [ 0,0,0,1,0,0,0,1]  &  192 &   \non & ? & \non  \\ \hline
\tvi A_{4}+A_{2} &  [ 0,0,0,0,0,2,0,0]  &  194  &   \non & ? & \non \\ \hline
\tvi A_{5} &  [ 2,0,0,0,0,1,0,1]  &  196  &   \non & \oui \leftarrow \O_{(3,2^{2},1^{7})}^{\{2,3,4,5,6,7,8\} }  & \non \\ \hline
\tvi D_{5}(a_{1})+A_{1} &  [ 0,0,0,1,0,0,0,2]  &  196  &   \non & ? & \non \\ \hline
\tvi A_{4}+A_{2}+A_{1} &  [ 0,0,1,0,0,1,0,0]  &  196  &   \non & ? & \non \\ \hline
\tvi D_{4}+A_{2} &  [ 0,2,0,0,0,0,0,2]  &  198  &   \non & \oui \leftarrow \O_{(2^{6},1^{2})}^{\{2,3,4,5,6,7,8\} } & \non \\ \hline
\tvi E_{6}(a_{3}) &  [ 2,0,0,0,0,0,2,0]  &  198  &   \non & \oui \leftarrow \O_{(3^{2},1^{8})}^{\{2,3,4,5,6,7,8\} } & \non  \\ \hline
\tvi D_{5} &  [ 2,0,0,0,0,0,2,2]  &  200 &   \non &  \oui \leftarrow \O_{(5,1^{9})}^{\{2,3,4,5,6,7,8\} } & \non  \\ \hline
\tvi A_{4}+A_{3} &  [ 0,0,0,1,0,0,1,0]  &  200 &   \non & ? & \oui \\ \hline
\tvi A_{5}+A_{1} &  [ 1,0,0,1,0,0,0,1]  &  202 &   \non & ? & \oui \\ \hline
\tvi D_{5}(a_{1})+A_{2} &  [ 0,0,1,0,0,1,0,1]  &  202 &   \non & ? & \oui \\ \hline
\tvi D_{6}(a_{2}) &  [ 0,1,1,0,0,0,1,0]  &  204  &   \non & ? & \non \\ \hline
\tvi E_{6}(a_{3})+A_{1} &  [ 1,0,0,0,1,0,1,0]  &  204  &   \non & ? & \non \\ \hline
\tvi E_{7}(a_{5}) &  [ 0,0,0,1,0,1,0,0]  &  206  &   \non & ? & \non \\ \hline
\tvi D_{5}+A_{1} &  [ 1,0,0,0,1,0,1,2]  &  208  &   \non & ? & \non \\ \hline
\tvi E_{8}(a_{7}) &  [ 0,0,0,0,2,0,0,0]  &  208  &   \non & \oui \leftarrow \O_{(3^{2},2^{2},1^{4})}^{\{2,3,4,5,6,7,8\} } & \non \\ \hline
\tvi A_{6} &  [ 2,0,0,0,0,2,0,0]  &  210 &   \non & \oui \leftarrow \O_{(3,1^{7}), (1^{3})}^{\{ 1,2,3,4,5,7,8\} } & \non \\ \hline
\tvi D_{6}(a_{1}) &  [ 0,1,1,0,0,0,1,2]  &  210 &   \non & \oui \leftarrow \O_{(3^{2},2^{4})}^{\{2,3,4,5,6,7,8\} } & \non \\ \hline
\tvi A_{6}+A_{1} &  [ 1,0,0,1,0,1,0,0]  &  212 &   \non & ? & \non \\ \hline
\tvi E_{7}(a_{4}) &  [ 0,0,0,1,0,1,0,2]  &  212 &   \non &  \oui \leftarrow \O_{[0,0,0,1,0,1,0]}^{\{1,2,3,4,5,6,7\} } & \non \\ \hline
\tvi E_{6}(a_{1}) &  [ 2,0,0,0,0,2,0,2]  &  214  &   \non & \oui \leftarrow \O_{(5,3,1^{6})}^{\{2,3,4,5,6,7,8\} } & \non
\end{array}
$$

\newpage

\subsection*{Type $\bf{E_{8}}$ (cont'd)}~

$$
\begin{Dynkin}
\Dspace\Dspace\Dbloc{\Dcirc\Dsouth\Dtext{t}{2}}
\Dskip
\Dbloc{\Dcirc\Deast\Dtext{b}{1}}
\Dbloc{\Dcirc\Dwest\Deast\Dtext{b}{3}}
\Dbloc{\Dcirc\Dwest\Deast\Dnorth\Dtext{b}{4}}
\Dbloc{\Dcirc\Dwest\Deast\Dtext{b}{5}}
\Dbloc{\Dcirc\Dwest\Deast\Dtext{b}{6}}
\Dbloc{\Dcirc\Dwest\Deast\Dtext{b}{7}}
\Dbloc{\Dcirc\Dwest\Dtext{b}{8}}
\end{Dynkin}
$$

$$
\begin{array}{c|c|c|c|c|c}
\multicolumn{2}{c|}{\O}&  \dim \O & RC_{1} &  RC_{2} & \hbox{rigid} \\ \hline
\tvi D_{5}+A_{2} &  [ 0,0,0,0,2,0,0,2]  &  214  &   \non & \oui \leftarrow \O_{(2^{3},1^{2})}^{\{ 1,3,4,5,6,7,8\} } & \non  \\ \hline
\tvi D_{6} &  [ 2,1,1,0,0,0,1,2]  &  216  &   \non & ? & \non \\ \hline
\tvi E_{6} &  [ 2,0,0,0,0,2,2,2]  &  216  &   \non & \oui \leftarrow \O_{(7,1^{7})}^{\{ 2,3,4,5,6,7,8\} } & \non \\ \hline
\tvi D_{7}(a_{2}) &  [ 1,0,0,1,0,1,0,1]  &  216  &   \non & \oui \leftarrow \O_{(1^{2}) ,(2^{2} , 1^{3})}^{\{ 1,2,4,5,6,7,8\} } & \non \\ \hline
\tvi A_{7} &  [ 1,0,0,1,0,1,1,0]  &  218  &   \non & ? & \non \\ \hline
\tvi E_{6}(a_{1})+A_{1} &  [ 1,0,0,1,0,1,0,2]  &  218  &   \non & \oui \leftarrow \O_{(4^{2},2^{2},1^{2})}^{\{ 2,3,4,5,6,7,8\} } & \non  \\ \hline
\tvi E_{7}(a_{3}) &  [ 2,0,0,1,0,1,0,2]  &  220 &   \non & \oui \leftarrow \O_{(5,3,2^{2},1)}^{\{ 2,3,4,5,6,7,8\} } & \non \\ \hline
\tvi E_{8}(b_{6}) &  [ 0,0,0,2,0,0,0,2]  &  220 &   \non & \oui \leftarrow \O_{(4^{3},3,1^{3})}^{\{ 2,3,4,5,6,7,8\} }  & \non \\ \hline
\tvi D_{7}(a_{1}) &  [ 2,0,0,0,2,0,0,2]  &  222 &   \non & \oui \leftarrow \O_{(4^{2},3^{2})}^{\{ 2,3,4,5,6,7,8\} } & \non \\ \hline
\tvi E_{6}+A_{1} &  [ 1,0,0,1,0,1,2,2]  &  222 &   \non & ? & \non \\ \hline
\tvi E_{7}(a_{2}) &  [ 0,1,1,0,1,0,2,2]  &  224  &   \non & ? & \non \\ \hline
\tvi E_{8}(a_{6}) &  [ 0,0,0,2,0,0,2,0]  &  224  &   \non & \oui \leftarrow \O_{(5,3^{3})}^{\{ 2,3,4,5,6,7,8\} } & \non  \\ \hline
\tvi D_{7} &  [ 2,1,1,0,1,1,0,1]  &  226  &   \non & ? & \non \\ \hline
\tvi E_{8}(b_{5}) &  [ 0,0,0,2,0,0,2,2]  &  226  &   \non & \oui \leftarrow \O_{(5^{2},2^{2})}^{\{ 2,3,4,5,6,7,8\} } & \non  \\ \hline
\tvi E_{7}(a_{1}) &  [ 2,1,1,0,1,0,2,2]  &  228  &   \non & \oui \leftarrow \O_{(7,3,2^{2})}^{\{ 2,3,4,5,6,7,8\} } & \non \\ \hline
\tvi E_{8}(a_{5}) &  [ 2,0,0,2,0,0,2,0]  &  228  &   \non & \oui \leftarrow \O_{(5^{2},3,1)}^{\{ 2,3,4,5,6,7,8\} } & \non \\ \hline
\tvi E_{8}(b_{4}) &  [ 2,0,0,2,0,0,2,2]  &  230 &   \non &  \oui \leftarrow \O_{(6^{2},1^{2})}^{\{ 2,3,4,5,6,7,8\}Ê} & \non \\ \hline
\tvi E_{7} &  [ 2,1,1,0,1,2,2,2]  &  232 &   \non & ? & \non \\ \hline
\tvi E_{8}(a_{4}) &  [ 2,0,0,2,0,2,0,2]  &  232 &   \non &  \oui \leftarrow \O_{(7,5,1^{2})}^{\{ 2,3,4,5,6,7,8\} } & \non   \\ \hline
\tvi E_{8}(a_{3}) &  [ 2,0,0,2,0,2,2,2]  &  234  &   \non & \oui \leftarrow \O_{(7^{2})}^{\{ 2,3,4,5,6,7,8\} }  & \non \\ \hline
\tvi E_{8}(a_{2}) &  [ 2,2,2,0,2,0,2,2]  &  236  &   \non & \oui \leftarrow \O_{(9,5)}^{\{ 2,3,4,5,6,7,8\} } & \non \\ \hline

\tvi E_{8}(a_{1}) &  [ 2,2,2,0,2,2,2,2]  &  238  &   \non & \oui \leftarrow \O_{(11,3)}^{\{ 2,3,4,5,6,7,8\} } & \non
\end{array}
$$


\begin{thebibliography}{............}

\bibitem[BF13]{BF} M. Brion and B. Fu, 
{\em Symplectic resolutions for conical symplectic varieties}, 
preprint arXiv:1310.8193, to appear in Internat. Math. Res. Notices. 

\bibitem[B98]{Bro} A. Broer, 
{\em Normal nilpotent varieties in $F_4$}, J. Algebra {\bf 207} (1998), n$^\circ$2, 
427--448.

\bibitem[CMo10]{CMo} J.-Y. Charbonnel and A. Moreau, 
{\em The index of centralizers of elements of reductive Lie algebras},  
Doc. Math., {\bf 15} (2010), 387--421.

\bibitem[CM93]{CMa} D. Collingwood and W.M. McGovern, 
{\em Nilpotent orbits in semisimple {L}ie algebras},
Van Nostrand Reinhold Co. New York {\bf 65} (1993). 

\bibitem[DEM]{DEM} T. De Fernex, L. Ein and M. Musta\c{t}\u{a}, 
{\em Vanishing theorems and singularities in birational
geometry}, book in preparation available at 
http://www.math.utah.edu/~defernex/book.pdf

\bibitem[EM09]{EM} L. Ein and M. Musta\c{t}\u{a}, 
{\em Jet schemes and singularities}, 
Proc. Sympos. Pure Math., {\bf 80}, Part 2, Amer. Math. Soc., Providence, (2009).

\bibitem[Ha76]{Ha} R. Hartshorne, 
{\em Algebraic Geometry}, Graduate Texts in Mathematics n$^\circ$52 (1977), 
Springer-Verlag, Berlin Heidelberg New York.

\bibitem[Hen14]{Hen} A. Henderson, 
{\em Singularities of nilpotent orbit closures}, 
preprint http://arxiv.org/pdf/1408.3888v1.pdf. 

\bibitem[Hes76]{He} W.H. Hesselink, 
{\em Cohomology and the resolution of the nilpotent variety},
Math. Ann. {\bf 223} (1976), 249--252.

\bibitem[Hi91]{Hi} V. Hinich, 
{\em On the singularities of nilpotent orbits}, 
Israel J. Math. {\bf 73} (1991), n$^\circ$3, 297--308. 

\bibitem[Is11]{Is} S. Ishii, 
{\em Geometric properties of jet schemes}, 
Comm. Algebra {\bf 39} (2011), n$^\circ$5, 1872--1882. 

\bibitem[K89]{Kr} H. Kraft, 
{\em Closures of conjugacy classes in $G_2$}, 
J. Algebra {\bf 126} (1989), n$^\circ$2, 454--465.

\bibitem[KP82]{KP} H. Kraft and C. Procesi, 
{\em On the geometry of conjugacy classes in classical groups}, 
Comment. Math. Helv. {\bf 57} (1982), n$^\circ$4, 539--602.

\bibitem[Ka06]{Ka} D. Kaledin, 
{\em Symplectic singularities from the Poisson point of view}, 
J. Reine Angew. Math. {\bf 600} (2006), 135--156.

\bibitem[Ke83]{Ke} G. Kempken, 
{\em Induced conjugacy classes in classical Lie algebras}, 
Abh. Math. Sem. Univ. Hamburg {\bf 53} (1983), 53--83.

\bibitem[Kol73]{Kol} E. Kolchin, 
{\em Differential algebra and algebraic groups}, 
Academic Press, New York 1973.

\bibitem[Kos63]{Ko} B. Kostant, 
{\em Lie group representations on polynomial rings}, 
Amer. J. Math. {\bf 85} (1963), 327--402.


\bibitem[LeS8]{LeS} T. Levasseur and S.P. Smith, 
{\em Primitive ideals and nilpotent orbits in type $G_2$}, 
J. Algebra {\bf 114} (1988), 81--105.

\bibitem[LuS79]{LS}G. Lusztig and N. Spaltenstein,
{\em Induced unipotent classes}, J. London Math. Soc.
{\bf 19} (1979), 41--52.


\bibitem[MFK94]{MFK} D. Mumford, J. Fogarty, and F. Kirwan, 
{\em Geometric invariant theory}, Third edition. Ergebnisse der Mathematik und ihrer Grenzgebiete (2), {\bf 34}, Springer-Verlag, Berlin, 1994. 

\bibitem[M01]{Mu} M. Musta\c{t}\u{a}, 
{\em Jet schemes of locally complete intersection canonical singularities}, 
Invent. Math., {\bf 145} (2001), n$^\circ$3, 397--424, 
with an appendix by D. Eisenbud and E. Frenkel. 


\bibitem[Nam13]{Na} Y. Namikawa, 
{\em On the structure of homogeneous symplectic varieties of complete intersection}, 
Invent. Math. {\bf 193} (2013), n$^\circ$1, 159--185.

\bibitem[Nas96]{Nas} J.F. Nash, 
{\em Arc structure of singularities}, A celebration of John F. Nash, Jr. 
Duke Math. J. {\bf 81} (1995), n$^\circ$1, 31--38 (1996).

\bibitem[P91]{Pa} D.I. Panyushev, 
{\em Rationality of singularities and the Gorenstein property for nilpotent orbits}, 
Functional Analysis and Its Applications, 1991, {\bf 25}, n$^\circ$3, 225--226.

\bibitem[RT92]{RT} M. Ra\"{i}s and P Tauvel, 
{\em Indice et polyn\^{o}mes invariants pour certaines alg\`{e}bres de Lie}, 
J. Reine Angew. Math. {\bf 425} (1992), 123--140. 

\bibitem[R74]{Ri} R.W. Richardson, 
{\em Conjugacy classes in parabolic subgroups of semisimple algebraic groups}, 
Bull. London Math. Soc. {\bf 6} (1974), 21--24. 

\bibitem[So03]{So} E. Sommers, 
{\em Normality of nilpotent varieties in $E_6$}. 
J. Algebra {\bf 270} (2003), n$^\circ$1, 288--306. 

\bibitem[Sp82]{Sp} N. Spaltenstein, 
{\em Classes Unipotentes et Sous-Groupes de Borel}, 
Number {\bf 946} in Lecture Notes in Math. Springer-Verlag, Berlin-Heidelberg-New York, 
1982.

\bibitem[TY05]{TY}
P. Tauvel and R.W.T. Yu, 
{\em Lie algebras and algebraic groups}, Monographs
in Mathematics (2005), Springer, Berlin Heidelberg New York.

\bibitem[W89]{We} J. Weyman, {\em The equations of conjugacy classes of nilpotent 
matrices}, Invent. Math. {\bf 98} (1989), n$^\circ$2, 229--245.

\bibitem[W02]{We2} J. Weyman, 
{\em Two results on equations of nilpotent orbits}, 
J. Algebraic Geom. {\bf 11} (2002), n$^\circ$4, 791--800.

\bibitem[W03]{We3} J. Weyman, 
{\em Cohomology of vector bundles and syzygies}, 
Cambridge Tracts in Mathematics, {\bf 149}. Cambridge University Press, 
Cambridge, 2003.
 
\bibitem[Y07]{Yu} C. Yuen, 
{\em Jet schemes of determinantal varieties}, 
Algebra, geometry and their interactions, 261--270, 
Contemp. Math., {\bf 448} (2007), Amer. Math. Soc., Providence, RI, 261--270. 

\end{thebibliography}
\end{document}